\documentclass[10pt]{article}

\usepackage{graphicx, caption}
\usepackage{amssymb}
\usepackage{amsthm}
\usepackage{amsmath}
\usepackage{enumitem}
\usepackage{framed} 
\usepackage{dsfont}

\usepackage{soul}
\usepackage{changepage}
\usepackage{comment}
\usepackage{float}
\usepackage[dvipsnames, table]{xcolor}
\definecolor{darkblue}{rgb}{0,0,0.5}
\usepackage[
    colorlinks=true,
    linkcolor=darkblue,
    citecolor=darkblue]{hyperref}
\usepackage{cleveref}
\usepackage{lscape}

\usepackage{tikz}
\usetikzlibrary{calc}
\usetikzlibrary{shapes}
\usetikzlibrary{decorations.pathreplacing,calligraphy}

\usepackage{mathtools}

\theoremstyle{plain}
\newtheorem{theorem}{Theorem}[section]
\newtheorem{lemma}[theorem]{Lemma}

\newtheorem{proposition}[theorem]{Proposition}
\newtheorem{claim}{Claim}[theorem]
\newtheorem{conjecture}[theorem]{Conjecture}
\newtheorem{observation}[theorem]{Observation}
\newtheorem{corollary}[theorem]{Corollary}
\theoremstyle{definition}
\newtheorem{definition}[theorem]{Definition}

\newtheorem*{remark}{Remark}

\usepackage{mdframed}
\usepackage{lipsum}

\newmdtheoremenv{definbox}[theorem]{Definition}

\newcommand\eps{\varepsilon}

\newcommand{\cG}{\ensuremath{\mathcal G}}

\newcommand{\Prob}[1]{\ensuremath{%
    \mathbb P\left[#1\right]
  }}
  \newcommand{\ProbCond}[2]{\ensuremath{%
    \mathbb P\left[#1\:\middle|\:#2\right]
  }}
\newcommand{\Expect}[1]{\ensuremath{%
    \mathbb E\left[#1\right]
  }}
  \newcommand{\aas}{a.a.s.\ }

\newcommand{\hamconnthresh}{\ensuremath{%
    \delta^{\mathrm{CON}}
}}
\newcommand{\rdthresh}{\ensuremath{%
    \delta^{\mathrm{RD}}
}}
\newcommand{\vsdthresh}{\ensuremath{%
    \delta^{\mathrm{VSD}}
}}

\newcommand{\cF}{\mathcal{F}}

\newcommand{\cH}{\mathcal{H}}

\newcommand{\doublesquig}{%
  \mathrel{%
    \vcenter{\offinterlineskip
      \ialign{##\cr$\rightsquigarrow$\cr\noalign{\kern-1.5pt}$\rightsquigarrow$\cr}%
    }%
  }%
}

\DeclareMathOperator{\gnp}{\cG}
\DeclareMathOperator{\gknp}{\cG^{(\emph{k})}}

\DeclareMathOperator{\maxdeg}{\Delta}

\usepackage{graphicx}%

\newcommand{\AM}[1]{\textcolor{purple}{ $\blacktriangleright$\ {\sf AM: #1}\
  $\blacktriangleleft$ }}
  
  \newcommand{\TK}[1]{\textcolor{ForestGreen}{ $\blacktriangleright$\ {\sf TK: #1}\
  $\blacktriangleleft$ }}

\AtBeginDocument{%
   \def\MR#1{}
} 

\title{Optimal spread for spanning subgraphs of Dirac  hypergraphs}
\author{
Tom Kelly\thanks{School of Mathematics, Georgia Institute of Technology, USA. Email: tom.kelly@gatech.edu.  Research supported by
the National Science Foundation under Grant No. DMS-224707} \and 
Alp Müyesser\thanks{Department of Mathematics, University College London, UK. Email: alp.muyesser.21@ucl.ac.uk} \and 
Alexey Pokrovskiy\thanks{Department of Mathematics, University College London, UK. Email: dralexeypokrovskiy@gmail.com}}

\begin{document}
\maketitle

\begin{abstract}

Let $G$ and $H$ be hypergraphs on $n$ vertices, and suppose $H$ has large enough minimum degree to necessarily contain a copy of $G$ as a subgraph.
We give a general method to randomly embed $G$ into $H$ with good ``spread''.
More precisely, for a wide class of $G$, we find a randomised embedding $f\colon G\hookrightarrow H$ with the following property: for every $s$, for any partial embedding $f'$ of $s$ vertices of $G$ into $H$, the probability that $f$ extends $f'$ is at most $O(1/n)^s$.
This is a common generalisation of several streams of research surrounding the classical Dirac-type problem. 
\par For example, setting $s=n$, we obtain an asymptotically tight lower bound on the number of embeddings of $G$ into $H$. This recovers and extends recent results of Glock, Gould, Joos, K\"uhn, and Osthus and of Montgomery and Pavez-Sign\'e regarding enumerating Hamilton cycles in Dirac hypergraphs. Moreover, using the recent developments surrounding the Kahn--Kalai conjecture, this result implies that many Dirac-type results hold robustly, meaning $G$ still embeds into $H$ after a random sparsification of its edge set. This allows us to recover a recent result of Kang, Kelly, K\"uhn, Osthus, and Pfenninger and of Pham, Sah, Sawhney, and Simkin for perfect matchings, and obtain novel results for Hamilton cycles and factors in Dirac hypergraphs.
\par Notably, our randomised embedding algorithm is self-contained and does not require Szemer\'edi's regularity lemma or iterative absorption.  
\end{abstract}

\section{Introduction}
A central theme in extremal graph theory is computing the minimum-degree thresholds for various properties. A prototypical result is Dirac's theorem~\cite{dirac1952}, which states that $n$-vertex graphs with minimum degree at least $n/2$ are Hamiltonian. This is of course tight, as graphs with minimum degree less than $n/2$ need not be connected. Results of this flavour, concerned with minimum-degree thresholds for containment of spanning structures, are often referred to as \textit{Dirac-type} theorems.

\par Although the assumption that $\delta(G)\geq n/2$ cannot be weakened in Dirac's theorem, we can say a lot more about $n$-vertex graphs with $\delta(G)\geq n/2$ (such graphs are henceforth referred to as Dirac graphs). S\'arközy, Selkow, and Szemer\'edi \cite{SSS2003} showed that Dirac graphs in fact contain $\Omega(n)^n$ many distinct Hamilton cycles, which is clearly optimal up to the value of the implied constant factor (the constant terms were later sharpened by Cuckler and Kahn \cite{CK2009}). We refer to results of this type as \textit{enumeration} results.

\par We also know that the Hamilton cycles in Dirac graphs can survive \textit{random edge deletions} (this phenomenon is often called \textit{robustness}, see the survey of Sudakov \cite{sudakov2017robustness}). Indeed, an influential result of Krivelevich, Lee, and Sudakov \cite{krivelevich2014robust} states that for any Dirac graph $G$, the random subgraph obtained by keeping each edge of $G$ independently with probability $p$ where $p\geq C\log n / n$ ($C$ is an absolute constant) is \textit{asymptotically almost surely} (abbreviated a.a.s.) Hamiltonian. We remark that this result is a common generalisation of Dirac's theorem, as well as the classical result of P\'osa~\cite{Po76} that the Erd\H{o}s-R\'enyi random graph $\gnp(n, C\log n / n)$ is \aas Hamiltonian. 

\par Recently, random graph theory was revolutionised by Frankston, Kahn, Narayanan, and Park's \cite{FKNP21} proof of the \textit{fractional expectation threshold vs.\ threshold} conjecture of Talagrand~\cite{Ta10} and Park and Pham's \cite{park2022proof}  proof of the even stronger Kahn--Kalai conjecture~\cite{KK07}.
A crucial corollary of Talagrand's conjecture relates the so-called \textit{spread measures} with thresholds. Roughly speaking, for the study of random graphs, this connection implies that if one can find a probability measure on copies of $G\subseteq K_n$ with good \textit{spread} (see Definition~\ref{defn:spread}), then $G$ \aas appears in $\gnp(n,p)$. 
Simply considering the uniform distribution on copies of $G\subseteq K_n$ already gives startlingly powerful results. An example that barely scratches the surface of this phenomenon is the following: If $G$ is a cycle on $n$ vertices, then the uniform distribution on copies of $G$ in $K_n$ (that is, on Hamilton cycles of $K_n$) 
turns out to be well-behaved enough that through Talagrand's conjecture, we can swiftly recover P\'osa's~\cite{Po76} result about Hamiltonicity of $\gnp(n, C\log n / n)$.


Two recent independent papers by Kang, Kelly, K\"uhn, Osthus, and Pfenninger (KKKOP)~\cite{KKKOP22} and Pham, Sah, Sawhney, and Simkin (PSSS)~\cite{PSSS22} demonstrated that \textit{spreadness results} can yield common generalisations of enumeration and robustness results.  
If there is a probability measure on copies of $G$ in $H$ with good spread, where $H$ is for example a hypergraph with minimum degree large enough to contain a copy of $G$, then Talagrand's conjecture gives for free that $H$ contains $G$ \textit{robustly}, in a similar sense to the aforementioned result of Krivelevich, Lee, and Sudakov (see the next subsection for precise formulations).  Moreover, in order for such a probability measure to have good spread, it must have large support; that is, there must be many copies of $G$ in $H$.

Calculating the spread of copies of $G\subseteq K_n$ sampled uniformly at random is in many cases quite approachable.
When $H$ is not a complete graph, copies of $G\subseteq H$ sampled uniformly at random are hard to reason about, and it is unclear how one can compute the spread of such a random subgraph.
Both KKKOP \cite{KKKOP22} and PSSS \cite{PSSS22} circumvent this difficulty by employing powerful methods from extremal/probabilistic combinatorics such as Szemer\'edi's regularity lemma and iterative absorption to construct a probability measure on the desired subgraphs with good spread. The contribution of the current paper is a simpler method that can achieve the same task, in many cases in significantly more generality. 

The notion of \textit{vertex spreadness}, introduced by PSSS~\cite{PSSS22}, is crucial.  Instead of considering a random copy of $G$ in $H$ and analyzing the probability that a subset of edges of $H$ appears in $G$, we consider a random \textit{embedding} of $G \hookrightarrow H$ and analyze the probability that a subset of vertices of $G$ are mapped to a subset of vertices of $H$.
Our main results provide a randomized algorithm for embedding $G \hookrightarrow H$ so that the vertex spreadness matches up to a constant factor the vertex spreadness of embedding $G \hookrightarrow K_n$ uniformly at random, when $G$ is a (hypergraph) cycle or the disjoint union of copies of some fixed hypergraph.
This allows us to recover and extend many results from \cite{KKKOP22, PSSS22}, as well as obtain novel enumeration results which extend recent results of Montgomery and Pavez-Sign\'e \cite{MPS23}. In the next subsections, we provide precise formulations of our results.

\subsection{Thresholds}
Before stating our results, we make precise some concepts already alluded to in the introduction.
For a $k$-uniform hypergraph $H$ and $d \in [k - 1]$, we let $\delta_d(H)$ be the minimum number of edges of $H$ that any set of $d$ vertices of $V(H)$ is contained in.

\begin{definition}[Minimum-degree/Dirac threshold]
\label{def:min_degree_threshold}
Let $\mathcal{F}$ be an infinite family of $k$-uniform hypergraphs. By $\delta_{\mathcal{F},d}$ we denote the smallest real number $\delta$ such that for all $\alpha>0$ and for all but finitely many $n$ the following holds. Let $n$ be such that there exists some $F\in \mathcal{F}$ with $|V(F)|=n$.
Let $H$ be an $n$-vertex $k$-uniform hypergraph with $\delta_d(H) \ge (\delta+\alpha)\binom{n-d}{k-d}$. Then, there exists some $F\in \mathcal{F}$ with $|V(F)|=n$ such that $H$ contains a copy of $F$.  
\end{definition}

Note that for all $\cF$ and $d$, the statement in Definition~\ref{def:min_degree_threshold} holds for $\delta = 1$ and for the infimum over all such $\delta\in[0,1]$, so $\delta_{\cF, d}$ is well-defined.



The notion of thresholds is central in the study of random graphs.  Here we present a general definition for the threshold for a binomial random $k$-uniform hypergraph to contain some spanning subgraph.

\begin{definition}
    Let $\mathcal{F}$ be an infinite family of $k$-uniform hypergraphs. Denote by $p=p^*_{\mathcal{F}}(n)$ the function so that for all $n\in \mathbb{N}$ (such that there exists some $F\in \mathcal{F}$ with $|V(F)|=n$), the binomial random hypergraph $\gknp(n,p)$ contains a copy of some $F\in \mathcal{F}$ with $|V(F)|=n$ with probability precisely $1/2$. 
\end{definition}

For example, if $\cF$ contains one length-$n$ cycle for every $n \in \mathbb N$, then $p^*_\cF(n)$ is the threshold for $\gnp(n, p)$ to have a Hamilton cycle.

\begin{definition}[$p$-random sparsification] Given $p\in[0,1]$ and some hypergraph $H$, by $H_p$, we denote the random graph obtained by keeping every edge of $H$ with probability $p$, making the decisions independently for each distinct edge of $H$.
\end{definition}

\begin{definition}[RD-threshold/Robust Dirac threshold] Let $\mathcal{F}$ be an infinite family of $k$-uniform hypergraphs, and let $p=p(n)\in[0,1]$ be a function. By $\rdthresh_{\cF, d, p}$, we denote the smallest real number $\delta$ such that for all $\alpha,\eps>0$, there exists $C>0$ such that for all sufficiently large $n\in \mathbb{N}$ (such that there exists some $F\in \mathcal{F}$ with $|V(F)|=n$), for any $n$-vertex hypergraph $H$ with $\delta_d(H)\geq (\delta+\alpha) \binom{n-d}{k-d}$, there exists some $F\in\mathcal{F}$ with $|V(F)|=n$ such that the random sparsification $H_{\hat{p}}$ contains a copy of $F$ with probability at least $1-\eps$ as long as ${\hat{p}}\geq C p(n)$. By $\rdthresh_{\mathcal{F},d}$, we denote $\rdthresh_{\mathcal{F},d, p^*_\mathcal{F}}$.
\end{definition}

It is clear that $\rdthresh_{\cF, d}\geq \delta_{\cF, d}$. The other inequality is a lot more interesting, and in the following cases, true.

\begin{theorem}\label{Prop_prev_results}
For the following values of $\cF$ and $d$, we have that $\rdthresh_{\cF, d}= \delta_{\cF, d}$:
\end{theorem}
\begin{enumerate}[label=(\arabic*)]
    \item\label{prev-result-hamiltonian} when $\mathcal{F}:=\{C_n\colon n\in\mathbb{N}\}$ where $C_n$ is a cycle on $n$ vertices (\textrm{Krivelevich, Lee, Sudakov} \cite{krivelevich2014robust});
    \item\label{prev-result-Kr-factor}for each $r\geq 2$, when $\mathcal{F}:=\{n \times K_r\colon n\in\mathbb{N}\}$ ($r$-clique-factors) (Pham, Sah, Sawhney, Simkin \cite{PSSS22}, the $r=3$ case was obtained by Allen, B\"ottcher, Corsten, Davies, Jenssen, Morris, Roberts, Skokan \cite{ABCDJMRS2022});
    \item\label{prev-result-spanning-tree} for each $C$, when $\mathcal{F}$ is made up of a single tree $T_n$ on $n$ vertices for each $n\in\mathbb{N}$, and $\Delta(T_n)\leq C$ (Pham, Sah, Sawhney, Simkin \cite{PSSS22});
    \item\label{prev-result-matchings} when $\mathcal{F} \coloneqq \{n \times K^{(k)}_k : n \in \mathbb N\}$ ($k$-uniform matchings), for any $d\in[k-1]$ (independently Pham, Sah, Sawhney, Simkin \cite{PSSS22} and Kang, Kelly, K\"uhn, Osthus, Pfenninger \cite{KKKOP22}).
\end{enumerate}

We remark that the first three results concern graphs, so $d=1$ in each case. Our main contribution in this paper is a unified strategy giving a far-reaching generalisation of the first, second, and fourth result stated above. In particular, this allows us to add hypergraph Hamilton cycles and $F$-factors for a much broader choice of $F$ (the second result concerns clique factors and the fourth result concerns perfect matchings, which are $K^{(k)}_k$-factors, where $K^{(k)}_n$ denotes the complete $k$-uniform $n$-vertex hypergraph) to the list above.  It is also conceivable that our methods work in even more generality than we explore in the current paper. 

We also remark that the first and the second result above hold in a stronger sense, with \textit{exact} minimum degree conditions as opposed to asymptotic ones, and Kang, Kelly, K\"uhn, Osthus, Pfenninger \cite{KKKOP22} obtained an exact result for the fourth result above in the case $d = k - 1$.
We refer the reader to the corresponding papers for more details.

\subsection{Hypergraph Hamilton cycles and Hamilton connectivity}\label{hamcycledefns}

We recall that for $0 \leq \ell <k$, a
$k$-uniform hypergraph is called an \textit{$\ell$-cycle} if its vertices can be ordered cyclically such that each of its
edges consists of $k$ consecutive vertices and every two consecutive edges (in the natural
order of the edges) share exactly $\ell$ vertices.
In particular, $(k-1)$-cycles and $1$-cycles are known as \textit{tight cycles} and \textit{loose cycles} respectively, and $0$-cycles are matchings. For $n\in (k-\ell)\mathbb{N}$, we let $C_{n, k, \ell}$ denote the $n$-vertex $\ell$-cycle with vertex set $[n]$ and edge set
\begin{equation*}
      \{[k + (k - \ell)(i- 1)]\setminus [(k - \ell)(i - 1)] : i \in [n / (k - \ell)]\},
\end{equation*}
where addition above is modulo $n$. We define
$k$-uniform $\ell$-paths $P_{n,k,\ell}$ analogously.

Let $\mathcal{C}_{\ell, k} \coloneqq \{C_{n,k,\ell} : n \in (k - \ell)\mathbb N\}$, and let $\mathcal{P}_{\ell, k} \coloneqq \{P_{n,k,\ell} : n \in (k - \ell)\mathbb N\}$.

Unlike in the setting of Theorem~\ref{Prop_prev_results}\ref{prev-result-Kr-factor} and Theorem~\ref{Prop_prev_results}\ref{prev-result-matchings}, we would like to find connected structures in our host graphs. We work with hypergraphs as opposed to graphs, so P\'osa-rotation-extension techniques as in Theorem~\ref{Prop_prev_results}\ref{prev-result-hamiltonian} are unavailable to us. For the same reason, relying on the regularity lemma (and related tools, such as the blow-up lemma) as in Theorem~\ref{Prop_prev_results}\ref{prev-result-spanning-tree} would be far from straight-forward. Instead, we use as a black box a connection between $\delta_{\mathcal{F}, d}$ and the following related parameter.

\begin{definition}[Hamilton connectivity threshold]
 By $\hamconnthresh_{k,\ell, d}$ we denote the smallest real number $\delta$ such that for all $\alpha>0$ and for all sufficiently large $n$ the following holds. For any $F\in \mathcal{P}_{\ell, k}$ with $|V(F)|=n$, for any $n$-vertex $k$-uniform hypergraph $H$ with $\delta_d(H)\geq (\delta+\alpha)\binom{n - d}{k - d}$, and any disjoint $S,T\subseteq V(H)$ with $|S|=|T|=\ell$, $H$ contains a copy of $F$ where the first and last $\ell$ vertices of $F$ (in the natural ordering of the vertices induced by the edges) are embedded to $S$ and $T$, respectively. 
\end{definition}
It is not hard to see that $\hamconnthresh_{k,\ell, d}\geq \delta_{\mathcal{C}_{\ell,k},d}$, and it is also known that $\hamconnthresh_{2,1, 1}=\delta_{\mathcal{C}_{1,2},1}$ (see for example \cite{DanielG}), but as far as we are aware, the parameter $\hamconnthresh_{k,\ell, d}$ has not been studied for other ranges of $k,\ell, d$.

\begin{proposition}\label{Prop_hamilton_connectivity}
    For each of the following range of parameters, we have that $\hamconnthresh_{k,\ell, d}=\delta_{\mathcal{C}_{\ell,k},d}$.
    \begin{enumerate}
            \item $d=k-1$;
            \item $1 \leq \ell < k/2$ and $d=k-2$;
            \item $\ell=k/2$ and $k/2 < d \le k-1$ with $k$ even.
    \end{enumerate}
\end{proposition}
The first part of the above proposition is proved as Lemma 2.6 in \cite{MPS23}. See \Cref{Sec_hamilton_connectivity} for the short proof of the second and third part of the proposition. These assertions are a consequence of the more general phenomenon that whenever there is a  ``standard absorption proof'' (as codified in \cite{gupta2022general}) that $\delta_{\mathcal{C}_{\ell,k},d}\leq \delta$, then $\hamconnthresh_{k,\ell, d}\leq \delta$ as well (see \Cref{prop:absorption-proof-gives-hamilton-connectivity}). Therefore, it is quite feasible that Proposition~\ref{Prop_hamilton_connectivity} remains true for a wider range of the parameters. But we remark that known results imply, for example, that $\hamconnthresh_{k,\ell, d}>\delta_{\mathcal{C}_{\ell,k},d}$ when $k=3$, $\ell=2$ and $d=1$ (see \cite{RRRSS19}).

\subsection{New robustness results for Hamilton cycles and factors}
An important consequence of our main theorem (Theorem~\ref{thm:main-vertex-spread}, to be stated in the next subsection) is that $\rdthresh_{\mathcal{C}_{\ell,k}, d} \leq \hamconnthresh_{k,\ell,d}$ for \textit{all} ranges of the parameters. This is implied by the combination of the following two theorems.

\begin{theorem}\label{thm:loose-cycle}
    For every $\alpha>0$ and $k \in \mathbb N$, there exists $C = C_{\ref*{thm:loose-cycle}}(\alpha, k)$ such that the following holds for every $d\in [k-1]$.  
    Let $H$ be an $n$-vertex $k$-uniform hypergraph such that $\delta_d(H)\geq (\hamconnthresh_{k,1, d}+\alpha)\binom{n - d}{k - d}$.
    If $(k - 1) \mid n$ and $p \geq  C \log n / n^{k - 1}$, then \aas a random subhypergraph $H_p$ contains a loose Hamilton cycle.
\end{theorem}

\begin{theorem}\label{thm:ell>1-cycle}
    For every $\alpha>0$ and $k \in \mathbb N$,  the following holds for every $d\in [k-1]$ and $\ell \in \{2, \dots, k - 1\}$.  
    Let $H$ be an $n$-vertex $k$-uniform hypergraph such that $\delta_d(H)\geq (\hamconnthresh_{k,\ell, d}+\alpha)\binom{n - d}{k - d}$.
    If $(k - \ell) \mid n$ and $p = \omega(1 / n^{k-\ell})$, then \aas a random subhypergraph $H_p$ contains a Hamilton $\ell$-cycle.
\end{theorem}
We chose to state these two theorems separately, to highlight the difference in the corresponding values of $p$ we need to work with for the cases of $\ell\leq 1$ and $\ell>1$ cases for the assertion that $\rdthresh_{\mathcal{C}_{\ell,k}, d} \leq \hamconnthresh_{k,\ell,d}$. Indeed, the function $p^*_{\mathcal{C}_{\ell,k}}$ is $\Theta( \log n / n^{k - 1})$ when $\ell\leq 1$ whereas it is $\Theta( 1 / n^{k - \ell})$ when $\ell>1$ (see \cite{DF11, DF13}). Roughly speaking, this difference is caused by isolated vertices being the bottleneck for random hypergraphs to contain perfect matchings or loose Hamilton cycles ($\ell\leq 1$), whereas for $\ell>1$, the bottleneck comes from a first-moment argument. Due to this subtlety, deriving Theorem~\ref{thm:ell>1-cycle} from Theorem~\ref{thm:main-vertex-spread} is more challenging. To circumvent this difficulty, we benefit from a recent generalisation of Spiro \cite{Sp21} of Talagrand's conjecture. 

\par Combining this with the connection between the parameters $\hamconnthresh_{k,\ell,d}$ and $\delta_{\mathcal{C}_{\ell,k},d}$, we derive the following, which considerably extends the families of hypergraphs listed in Theorem~\ref{Prop_prev_results}. 

\begin{corollary} For each $k,\ell$ and $d$ such that $\hamconnthresh_{k,\ell, d}=\delta_{\mathcal{C}_{\ell,k},d}$, we have that $\rdthresh_{\mathcal{C}_{\ell,k}, d}= \delta_{\mathcal{C}_{\ell,k}, d}$. In particular, for the ranges of $k,\ell$ and $d$ described in Proposition~\ref{Prop_hamilton_connectivity}, we have that  $\rdthresh_{\mathcal{C}_{\ell,k}, d}= \delta_{\mathcal{C}_{\ell,k}, d}$.
\end{corollary}

Independently of our work, Anastos, Chakraborti, Kang, Methuku, and Pfenninger \cite{anastos2023} also obtained that $\rdthresh_{\mathcal{C}_{\ell,k}, d}= \delta_{\mathcal{C}_{\ell,k}, d}$ in the regime where $d=k-1$ and $\ell<k/2$. 
In fact, in this regime they obtained a stronger result working with \textit{exact} minimum degree thresholds.
We conjecture that the hypothesis $\hamconnthresh_{k,\ell, d}=\delta_{\mathcal{C}_{\ell,k},d}$ is not required in the above statement. See Section~\ref{sec:openproblems} for more details. 

For a hypergraph $F$, let $\cF_F$ be a family of disjoint unions of copies of $F$ containing one element with $i|V(F)|$ vertices for every $i \in \mathbb N$, and let $\delta_{F,d} \coloneqq \delta_{\cF_F, d}$.  Hence, if an $n$-vertex hypergraph $H$ satisfies $\delta_d(H) \geq (\delta_{F,d} + \alpha)\binom{n - d}{k-d}$ and $n$ is sufficiently large and divisible by $|V(F)|$, then $H$ contains an \textit{$F$-factor}; that is, $H$ contains a spanning subgraph where every component is a copy of $F$.  
Our next theorem implies $\rdthresh_{\cF_F, d} = \delta_{F, d}$ when $F$ is a \textit{strictly $1$-balanced} uniform hypergraph.    
For a hypergraph $H$ with at least two vertices, the \textit{$1$-density} of $H$ is $d_1(H) \coloneqq |E(H)| / (|V(H)| - 1)$.  A hypergraph $F$ is \textit{$1$-balanced} if $d_1(F') \leq d_1(F)$ for every $F' \subseteq F$ and \textit{strictly $1$-balanced} if $d_1(F') < d_1(F)$ for every $F' \subsetneq F$.
For strictly $1$-balanced $k$-uniform $F$, Johansson, Kahn, and Vu~\cite{JKV08} proved that the threshold for a binomial random $k$-uniform hypergraph to contain an $F$-factor is $\Theta(n^{-1 / d_1(F)}\log^{1 / |E(F)|}n)$.  
  
\begin{theorem}\label{thm:F-factors-robustness}
    For every $\alpha>0$ and $k,r \in \mathbb N$, there exists $C = C_{\ref*{thm:F-factors-robustness}}(\alpha, k, r) > 0$ such that the following holds for every $d\in [k-1]$.  
    Let $F$ be a $k$-uniform $r$-vertex strictly $1$-balanced hypergraph, and let $H$ be an $n$-vertex $k$-uniform hypergraph such that $\delta_d(H)\geq (\delta_{F,d}+\alpha)\binom{n - d}{k - d}$.
    If $r \mid n$ and $p \geq C n^{-1 / d_1(F)}\log^{1 / |E(F)|}n$, then \aas a random subhypergraph $H_p$ contains an $F$-factor.
\end{theorem}

Complete graphs are strictly $1$-balanced, so Theorem~\ref{thm:F-factors-robustness} implies Theorem \ref{Prop_prev_results}\ref{prev-result-Kr-factor}.  A $k$-uniform $k$-vertex hypergraph consisting of a single edge is also strictly $1$-balanced, so Theorem~\ref{thm:F-factors-robustness} implies Theorem~\ref{Prop_prev_results}\ref{prev-result-matchings}.

\subsection{Vertex-spread thresholds}
The following definition, originally introduced in \cite{PSSS22}, is at the heart of the paper.
\begin{definition}
    Let $X$ and $Y$ be finite sets, and let $\mu$ be a probability distribution over injections $\varphi : X \rightarrow Y$.  
    For $q \in [0, 1]$, we say that $\mu$ is \textit{$q$-vertex-spread} if for every two sequences of distinct vertices $x_1, \dots, x_s \in X$ and $y_1, \dots, y_s \in Y$,
    \begin{equation*}           
    \mu\left(\left\{\varphi : \varphi(x_i)=y_i \text{ for all }i \in [s]\right\}\right) \leq q^s.
    \end{equation*}    
\end{definition}

A \textit{hypergraph embedding} $\varphi : G \hookrightarrow H$ of a hypergraph $G$ into a hypergraph $H$ is an injective map $\varphi : V(G) \rightarrow V(H)$ that maps edges of $G$ to edges of $H$, so there is an embedding of $G$ into $H$ if and only if $H$ contains a subgraph isomorphic to $G$.  
Note that the uniformly random embedding $\varphi : G\hookrightarrow H$ when $G$ and $H$ have the same vertex set and $H$ is complete is just a uniformly random permutation of $V(H)$, which is $(e / |V(H)|)$-vertex-spread (using Stirling's approximation).  

\begin{definition}[VSD-threshold/Vertex-spread Dirac threshold] Let $\mathcal{F}$ be an infinite family of $k$-uniform hypergraphs. By $\vsdthresh_{\mathcal{F},d}$ we denote the smallest real number $\delta$ such that for all $\alpha>0$, there exists $C > 0$ such that the following holds for all but finitely many $F\in \mathcal{F}$.
If $n=|V(F)|$ and $H$ is any $n$-vertex $k$-uniform hypergraph with $\delta_d(H) \ge (\delta+\alpha)\binom{n-d}{k-d}$, then there is a $(C/n)$-vertex-spread distribution on embeddings of $F\hookrightarrow H$.
\end{definition}

Note that for every family $\mathcal{F}$, we have that $\delta_{\mathcal{F},d} \leq \rdthresh_{\mathcal{F},d}, \vsdthresh_{\mathcal{F},d} \leq 1 $.


The main results of this paper are that for every $k\in\mathbb N$,
\begin{enumerate}
    \item $\rdthresh_{\mathcal C_{k,\ell}, d} \leq \vsdthresh_{\mathcal{C}_{k,\ell},d} \leq \hamconnthresh_{k,\ell, d}$ for every $\ell, d \in [k-1]$ and 
    \item $\vsdthresh_{\cF_F, d} = \delta_{F, d}$ for every $k$-uniform $F$ and $d \in [k - 1]$ and moreover $\rdthresh_{\cF_F, d} = \delta_{F, d}$ if $F$ is strictly $1$-balanced.
\end{enumerate}
The following theorem implies $\vsdthresh_{\mathcal{C}_{k,\ell},d} \leq \hamconnthresh_{k,\ell, d}$, for all ranges of the parameters.
\begin{theorem}\label{thm:main-vertex-spread} 
For every $\alpha>0$ and $k \in \mathbb N$, there exists $C = C_{\ref*{thm:main-vertex-spread}}(\alpha, k)$ such that the following holds for every $\ell, d\in [k-1]$, and every sufficiently large $n$ for which $k - \ell$ divides $n$.  If $H$ is an $n$-vertex $k$-uniform hypergraph such that $\delta_d(H)\geq (\hamconnthresh_{k,\ell, d}+\alpha)\binom{n - d}{k - d}$, then there is a $(C/n)$-vertex-spread distribution on embeddings $C_{n,k,\ell} \hookrightarrow H$.
\end{theorem}


Note a consequence of the above theorem is that whenever, $\hamconnthresh_{k,\ell, d}= \delta_{\mathcal{C}_{k,\ell}, d}$, we have that $\hamconnthresh_{k,\ell, d}= \delta_{\mathcal{C}_{k,\ell}, d}= \vsdthresh_{\mathcal{C}_{k,\ell},d}$. 

The following theorem implies $\vsdthresh_{\cF_F, d} = \delta_{F, d}$ for every $k$-uniform $F$ and $d \in [k - 1]$.

\begin{theorem}\label{thm:vtx-spread-for-factors}
    For every $\alpha > 0$ and $k,r \in \mathbb N$ there exists $C=C_{\ref*{thm:vtx-spread-for-factors}}(\alpha, k, r)>0$ such that the following holds for all $d \in [k-1]$ and all sufficiently large $n$ for which $r$ divides $n$. Let $F$ be a $k$-uniform $r$-vertex hypergraph, and let $G$ be the union of $n / r$ disjoint copies of $F$. 
    If $H$ is an $n$-vertex $k$-uniform hypergraph such that $\delta_d(H) \geq (\delta_{F, d} + \alpha)\binom{n - d}{k - d}$, then there is a $(C / n)$-vertex-spread distribution on embeddings $G \hookrightarrow H$.
\end{theorem}

\subsubsection{The enumeration aspect}

Note that if there is a $(C / n)$-vertex-spread distribution on embeddings $G \hookrightarrow H$, then
\begin{equation*}
1 = \sum_{\varphi : G \hookrightarrow H}\mu(\{\varphi\}) \leq \left|\left\{\varphi : G \hookrightarrow H\right\}\right|\left(\frac{C}{n}\right)^n = \left|\left\{\varphi : G \hookrightarrow H\right\}\right|\exp\left(n\log C - n\log n\right),
\end{equation*}
so there are at least $\exp(n\log n - O(n))$ embeddings and at least $\exp(n\log n - O(n)) / |\mathrm{Aut}(G)|$ copies of $G$ in $H$.  In particular, Theorems~\ref{thm:main-vertex-spread} and \ref{thm:vtx-spread-for-factors} imply that under their respective assumptions, the hypergraph $H$ contains at least 
\begin{itemize}
  \item $\exp\left(n \log n - O(n)\right)$ Hamilton $\ell$-cycles (for $\ell\neq0$) and
  \item $\exp\left((1 - 1/|V(F)|)n\log n - O(n)\right)$ $F$-factors,
\end{itemize}
respectively, which strengthen recent results of Montgomery and Pavez-Sign\'e~\cite{MPS23} and Glock, Gould, K\"{u}hn, and Osthus~\cite{GGJKO2021}. This statement can be interpreted as a far-reaching generalisation of the result of  S\'arközy, Selkow, and Szemer\'edi \cite{SSS2003} stated in the introduction.
\par We also remark that Ferber, Hardiman, and Mond~\cite{FHM2021} sharpened the asymptotics for $\ell$-cycles for $\ell < k - 1$ and $d = k - 1$ by determining the leading constant in the $O(n)$ term. We refer the reader to \cite{MPS23} for more in depth discussion of the enumeration aspect of Dirac-type results.    
\subsection{Spreadness}

As previously mentioned, deriving Theorem~\ref{thm:loose-cycle} and Theorem~\ref{thm:ell>1-cycle} from Theorem~\ref{thm:main-vertex-spread} requires us to pass through the fractional expectation thresholds breakthrough of Frankston, Kahn, Narayanan, and Park \cite{FKNP21}, and in fact, a generalisation of this breakthrough by Spiro~\cite{Sp21}. A hypergraph is called $r$-\textit{bounded} if each edge has size at most $r$.

\begin{definition}\label{defn:spread}
    Let $q \in [0, 1]$, and let $r\in \mathbb N$. Let $(V,\cH)$ be an $r$-bounded hypergraph, and let $\mu$ be a probability distribution on $\cH$.  We say $\mu$ is \textit{$q$-spread} if the following holds:
\begin{equation*}
    \mu\left(\left\{A \in \cH : A \supseteq S\right\}\right) \leq q^{|S|} \text{ for all } S \subseteq V.
\end{equation*}
\end{definition}

Frankston, Kahn, Narayanan, and Park (FKNP)~\cite{FKNP21} proved that if $(V, \cH)$ is an $r$-bounded hypergraph and $\cH$ supports a $q$-spread distribution, then a $p$-random subset of $V$ contains an edge in $\cH$ \aas if $p \geq K q\log r$ as $r \rightarrow \infty$ (here, $K$ is some absolute constant).
We will consider hypergraphs of the form $(V, \cH)$ where $V \coloneqq E(H)$ and $\cH$ is the set of (edge sets of) Hamilton $\ell$-cycles in some hypergraph $H$.
In this case, we sometimes abuse notation and write $\mu(A)$ instead of $\mu(\{A\})$ where $A \in E(H)$ or $\mu(F)$ instead of $\mu(\{E(F)\})$ when $F \subseteq H$.

Let $m_1(H) = \max_{H'\subseteq H: |V(H')|>1}d_1(H)$. The following proposition allows us to connect spread distributions and vertex-spread distributions.

\begin{proposition}\label{prop:vtx-spread-implies-spread}
    For every $C, k, \Delta > 0$, there exists $C'=C_{\ref*{prop:vtx-spread-implies-spread}}(C, k, \Delta)>0$ such that the following holds for all sufficiently large $n$. 
    Let $H$ and $G$ be $n$-vertex $k$-uniform hypergraphs. If there is a $(C / n)$-vertex-spread distribution on embeddings $G \hookrightarrow H$ and $\maxdeg(G) \leq \Delta$, then there is a $\left(C' / n^{1/m_1(G)}\right)$-spread distribution on subgraphs of $H$ which are isomorphic to $G$.  
\end{proposition}

In the case when $H$ is complete, since the uniform distribution on embeddings of $G$ into $H$ is $(e / |V(H)|)$-vertex-spread, Proposition~\ref{prop:vtx-spread-implies-spread} provides an upper bound on the spreadness for copies of $G$ in $H$.  For many natural choices for $G$ (such as in the case of $F$-factors or Hamilton $\ell$-cycles), this bound is in fact best possible due to the duality between spreadness and fractional expectation thresholds.

Moreover, Proposition~\ref{prop:vtx-spread-implies-spread} implies $\rdthresh_{\mathcal C_{k,\ell}, d} \leq \vsdthresh_{\mathcal{C}_{k,\ell},d}$ for every $\ell \in \{0,1\}$ and $d \in [k-1]$.
Combined with Theorem~\ref{thm:vtx-spread-for-factors}, it yields a $(O(1) / n^{k-1})$-spread distribution on perfect matchings in sufficiently dense $n$-vertex $k$-uniform hypergraphs, which was proved independently by KKKOP \cite[Theorem 1.5]{KKKOP22} and PSSS \cite[Theorem 1.5]{PSSS22} (in the $k \mid n$ case).
Proposition~\ref{prop:vtx-spread-implies-spread} combined with Theorem~\ref{thm:main-vertex-spread} also yields a $(O(1) / n^{k-1})$-spread distribution on loose Hamilton cycles (see Lemma~\ref{lemma:ell-cycle-subgraph-bound} for $m_1(C_{n,k,1})$) in sufficiently dense $n$-vertex $k$-uniform hypergraphs.  Combined with the FKNP theorem, this result implies Theorem~\ref{thm:loose-cycle}.

To prove Theorem~\ref{thm:ell>1-cycle}, we need Spiro's~\cite{Sp21} strengthening of the FKNP theorem -- see Section~\ref{sect:spiro-spread}.
To prove Theorem~\ref{thm:F-factors-robustness}, we combine Proposition~\ref{prop:vtx-spread-implies-spread} with a coupling argument of Riordan~\cite{Ri22} -- see Section~\ref{sect:coupling}.

\begin{remark}
    In independent work, Joos, Lang, and Sanhueza-Matamala \cite{JLS-M23} also proved Theorems \ref{thm:loose-cycle}, \ref{thm:ell>1-cycle}, and \ref{thm:F-factors-robustness}.  They also obtained these results from stronger results concerning spreadness, but they did not consider vertex spreadness.  To prove Theorem \ref{thm:ell>1-cycle}, they used a result of Espuny D\'{\i}az and Person \cite{E-DP23}, which is generalized by the result of Spiro \cite{Sp21}.
\end{remark}



\section{Preliminaries}

We use standard notation for ``hierarchies'' of constants, writing $x\ll y$ to mean that there is a non-decreasing function $f : (0,1] \rightarrow (0, 1]$ such that the subsequent statements hold for $x\leq f(y)$. Hierarchies with multiple constants are defined similarly.  We omit rounding signs where they are not crucial.  

We need the following well-known version of the Chernoff bounds. 

\begin{lemma}[Chernoff bound]\label{chernoff} Let $X:=\sum_{i=1}^m X_i$ where $(X_i)_{i\in[m]}$ is a sequence of independent indicator random variables, and let $\Expect{X}=\mu$. For every $\gamma \in (0, 1)$, we have $\Prob{|X-\mu|\geq \gamma \mu}\leq 2e^{-\mu \gamma^2/3}$.
\end{lemma}


We will also use a result of Gupta, Hamann, M{\"u}yesser, Parczyk, and Sgueglia~\cite{gupta2022general}, which is proved with a standard concentration inequality.

\begin{lemma}[\cite{gupta2022general}, Lemma 3.5]\label{lem:McDAppl}
Let $k,\ell,d \in \mathbb{N}$, $0<\delta' <\delta< 1$ and $1/n, 1/\ell \ll 1/k, \delta-\delta'$.
Let $H$ be a $k$-uniform $n$-vertex hypergraph with vertex set $V$ and suppose that $\deg(D,V) \geq \delta n^{k-d}$ for each $D \in \binom{V}{d}$. If $A\subseteq V$ is a vertex set of size $\ell$ chosen uniformly at random, then for every $D \in \binom{V}{d}$ we have 

\begin{equation*}
\Prob{\deg(D,A) < \delta' \ell^{k-d}}\leq 2 \exp (-\ell(\delta - \delta')^2/2).
\end{equation*}
\end{lemma}

Finally, we need the following result due to McDiarmid, which appears in the textbook of Molloy and Reed~\cite[Chapter 16.2]{MR02}.
Here, a \textit{choice} is the position that a particular element gets mapped to in a permutation.
\begin{lemma}[McDiarmid's inequality for random permutations]\label{lem:mcdiarmidperm}Let $X$ be a non-negative random variable determined by a random permutation $\pi$ of $[n]$ such that the following holds for some $c,r>0$:
\begin{enumerate}
    \item Interchanging two elements of $\pi$ can affect the value of $X$ by at most $c$
    \item For any $s$, if $X\geq s$ then there is a set of at most $rs$ choices whose outcomes certify that $X\geq s$.
\end{enumerate}
Then, for any $0\leq t\leq \mathbb{E}[X]$,
 
$$\mathbb{P}(|X-\mathbb{E}[X]|\geq t + 60c\sqrt{r\mathbb{E}[X]})\leq 4\exp(-t^2/(8c^2r\mathbb{E}[X])).$$
    
\end{lemma}

\section{Proof overview}\label{sect:overview}
In this section we overview the proofs of Theorems~\ref{thm:main-vertex-spread} and \ref{thm:vtx-spread-for-factors}.  The proof of Theorem~\ref{thm:vtx-spread-for-factors} is simpler, so we begin with this one.  In fact, for the sake of simplicity, suppose $k=2$, $d=1$ and $F$ is just a triangle, and the host graph $H$ is a graph on $n$ vertices, $n$ is divisible by $3$, and $\delta(H)\geq (2/3+\alpha)n$.   Allen, B\"ottcher, Corsten, Davies, Jenssen, Morris, Roberts, Skokan \cite{ABCDJMRS2022} recently proved Theorem~\ref{thm:F-factors-robustness} for this case (in fact, they proved the stronger result for $\alpha = 0$), but our result (in the $\alpha > 0$ case) is stronger in that it concerns vertex spreadness and also has a shorter proof.  Thus, this special case is still of independent interest, and our proof generalises easily to other graph factors (additional ideas are required to handle connected spanning subgraphs, as we will explain later).  For now, let us also assume that $n$ is divisible by some integer $C$ that is in turn divisible by $3$.

Our goal is to embed a triangle factor in $H$ in a $O(1/n)$-vertex-spread manner. The main idea is quite simple: we wish to first partition $V(H)$ into random ``clusters'' $U_1,\ldots, U_m$ each inducing a graph with good minimum degree, and then embed a triangle factor in each cluster $U_i$ in a deterministic fashion. For the latter task, we hope to rely only on $\delta(U_i)\geq (2/3+\alpha/2)|U_i|$ and $|U_i|$ being divisible by $3$ as a black box to argue that an embedding exists. This step is not randomised, and thus gains us no spreadness whatsoever, so we have to rely exclusively on the randomness of the partition to obtain spreadness on this simple algorithm. Let us first suppose that the partition is chosen uniformly at random given that each $U_i$ has size $C$.  A simple calculation shows that for every set of distinct vertices $x_1, \ldots, x_s\in V(H)$ and every function $f\colon [s]\to [m]$,
    \begin{equation}\label{sketch}
        \Prob{x_i\in U_{f(i)} \text{ for each }i\in [s]}\leq \left(\frac{C}{n}\right)^s.
    \end{equation} 

Suppose that with probability at least $1/2$, each $H[U_i]$ contains a triangle factor. Then, \eqref{sketch} is actually all we need, and by conditioning on each $H[U_i]$ containing a triangle factor, we obtain a $(2C/ n)$-vertex-spread distribution on embeddings of triangle factors in $H$.  

If $C\gg  \log n$, then it is straightforward to show that indeed $\delta(H[U_i])\geq (2/3+\alpha/2)|U_i|$ for each $i$ with probability at least $1/2$ by an application of Chernoff's bound and a union bound over the clusters.  By the Corr\'{a}di--Hajnal theorem~\cite{CH1963}, each $H[U_i]$ contains a triangle factor, as desired.  
This argument already implies the desired result up to a logarithmic factor, but to obtain the desired result, we need $C = O(1)$.  In this regime, it may be the case that with high probability, a small percentage of the $H[U_i]$ contain isolated vertices, and these clusters cannot possibly contain triangle factors. 

A central idea in the proof, replacing the more intricate iterative absorption strategies employed in \cite{KKKOP22, PSSS22}, is a \textit{random redistribution} argument. If $C$ is a large constant, it will still be the case that most of the $U_i$ have good minimum degree ($\delta(H[U_i])\geq (2/3+\alpha/2)|U_i|$); in this case, we call the cluster \textit{good} and otherwise call it \textit{bad}.  The key idea is that we can randomly redistribute vertices of the bad clusters to the good ones while preserving \eqref{sketch} and the minimum degree property. 

With high probability, it will be the case that each vertex $v$ can be added to many other good clusters $U_i$ while ensuring that $H[U_i\cup\{v\}]$ still has good minimum degree. One way to randomly redistribute would be the following. For each vertex $v\in V(H)$ living in a bad cluster, among the good clusters $U_i$ so that $U_i+v$ is also good, choose some $U_i$ uniformly at random, and re-define $U_i \coloneqq U_i+v$. This actually would not break \eqref{sketch}, but it would cost us the property that each $|U_i|$ is divisible by three, which is necessary for finding a triangle factor.
More importantly, it would also cost us the property that each random set has the same size, which is surprisingly critical while showing that the corresponding random embedding is vertex-spread (see the proof of Lemma~\ref{lemma:distribution-is-spread}, in particular, Observation~\ref{obs:windows}).

Hence, we would like to do the redistribution step while maintaining some control over the sizes of the clusters. We achieve this by ensuring each good cluster receives exactly one new element by adding the following step to the algorithm. (Thus, in the triangle factor case, we initially choose each $U_i$ to satisfy $|U_i| \equiv 2\mod 3$).  We first define an auxiliary bipartite graph $B$ between vertices $v$ which need to be redistributed and good clusters $U_i$, having $v\sim_B U_i$ only when $U_i + v$ also has good minimum degree. By potentially redistributing more vertices than necessary, we can ensure that this is a balanced bipartite graph, and it is not too difficult to check that $B$ will have high enough minimum degree to contain a perfect matching with high probability. We want to redistribute randomly, in particular we want a redistribution that preserves \eqref{sketch}. So in $B$, we are looking for a random perfect matching $M$ so that the probability that $M$ extends a given matching of size $s$ is $O(1/n)^s$. In other words, we want to find a $O(1/n)$-spread (not vertex-spread) probability distribution on  perfect matchings in $B$. Such a distribution on perfect matchings does exist, and in fact this assertion is just a weaker version of results already present in \cite{PSSS22}. However the statement we need here has such a short proof that we include the simple outline below (a similar proof due to Pham is recorded in an earlier version of the paper \cite{PSSS22} available on arXiv).
\begin{lemma}[Pham, Sah, Sawhney, Simkin \cite{PSSS22}]\label{lem:graphmatchings}
There exists an absolute constant $C_{\ref*{lem:graphmatchings}}$ with the following property. If $G$ is a balanced bipartite graph on $2n$ vertices with $\delta(G)\geq 3n/4$, then there exists a $(C_{\ref*{lem:graphmatchings}}/n)$-spread distribution on perfect matchings of $G$. 
\end{lemma}
\begin{proof}[Proof sketch]
    Consider a random subgraph of $G'$ obtained by sampling $C$ edges incident to each vertex of $G$ uniformly at random (we allow edges to be selected twice). We claim that with probability at least $99/100$, $G'$ contains a perfect matching. One can easily verify this by checking Hall's condition for $G'$. Now, condition $G'$ on the event that it satisfies Hall's condition, and consider an arbitrary perfect matching of $G'$. This defines $M$, which is a random perfect matching of $G$. We claim the distribution of $M$ is $O(1/n)$-spread. Indeed, consider a matching of $M'$ of $G$ of size $s$. Each edge $e$ of $M'$ can be included in $M$ only if $e=\{x,y\}$ is sampled from the side of $x$ or $y$, this event has probability at most $2C/(3n/4)=O(1/n)$. Furthermore, for disjoint edges $e$ and $e'$, these events are independent. This implies that $M$ extends $M'$ with probability at most $O(1/n)^s$, implying the desired spreadness.
\end{proof}

In Section~\ref{sec:randomcluster}, the random partitioning aspect of the proof is treated rigorously, culminating in Lemma~\ref{lem:randompartitioninglemma}, the ``random clustering lemma''. This lemma might be of independent interest, and can potentially be used as a black box to solve adjacent problems. There are two aspects of Lemma~\ref{lem:randompartitioninglemma} that we have not discussed. Firstly, we cannot assume that $n$ is divisible by $C$ for some constant $C=O(1)$, and therefore, we cannot guarantee that \textit{all} clusters in our partition have the same size. We handle this difficulty by having a single exceptional cluster which has a different size, and luckily, it turns out that this is good enough for the applications. 
For the triangle-factor case, we can use that, if all of the non-exceptional clusters $U_i$ have size divisible by three, then so does the exceptional one.  This argument generalises and is sufficient for the proof of Theorem~\ref{thm:vtx-spread-for-factors}.

The more important aspect of Lemma~\ref{lem:randompartitioninglemma} that we have not discussed is that in Theorem~\ref{thm:main-vertex-spread} (and possibly in future applications), we want to find structures which are connected (such as Hamilton cycles), so we need to think about edges that go across the clusters. To reason about this, consider an auxiliary graph $Q$ where vertices represent clusters $U_i$, and we have $U_i\sim_Q U_j$ whenever the bipartite graph between $U_i$ and $U_j$ has good minimum degree. Before the redistribution, $Q$ itself may have many isolated vertices.  Thus, in the redistribution step, this will be another aspect we have to consider, so that we can ensure that $Q$ itself is well-connected. 
If $Q$ itself has a Hamilton cycle, then we can use this cycle to reorder the clusters, and use this new ordering to embed a Hamilton $\ell$-cycle in $H$ path-by-path. That is, we can first find a Hamilton path in each cluster, and then connect them together along a Hamilton cycle of $Q$ to find the desired Hamilton cycle. See Figure~\ref{fig:embedding-algorithm}.

In fact, we need access to \textit{many} Hamilton cycles of $Q$. Indeed, if we were only working with a fixed Hamilton cycle of $Q$, we could guarantee at most $C!^{n/C}$ Hamilton cycles for any choice of $U_1,\ldots, U_m$, since each $U_i$ of size $C$ has at most $C!$ Hamilton paths, and we have at most $n/C$ random sets. Also, it is easy to see there are at most $n!/(n/C)!$ choices for $U_1,\ldots, U_m$ (here we consider two choices the same if they are identical up to relabelling the $U_i$). Thus, in total, this approach could yield at most $C!^{n/C} n!/(n/C)!$ distinct Hamilton cycles.  On the other hand, the definition of $O(1/n)$-vertex-spread requires us to find at least $(n/C)^n$ distinct Hamilton cycles, which is much larger than the previous quantity, as a simple application of Stirling's approximation shows. 

More thought reveals that to maintain vertex spreadness in the algorithm, we need a vertex-spread distribution on Hamilton cycles of $Q$. This may seem a bit circular, but it will be the case that $Q$ is almost complete, so the Dirac condition is satisfied with lots of room to spare. Furthermore, the problem we reduce to is a problem about graphs, whereas our general theorem concerns hypergraphs. The toolkit we have to find Hamilton cycles on graphs is significantly more extensive compared to hypergraphs. To conclude, to find a good ordering on $Q$, we rely on the following result. 
\begin{lemma}\label{lem:graphhamcycles}
    There exists an absolute constant $C_{\ref*{lem:graphhamcycles}}$ with the following property. If $G$ is a graph with $\delta(G)\geq 3n/4$ and $e(G)\geq (1-1/C_{\ref*{lem:graphhamcycles}})\binom{n}{2}$, then there exists a $(C_{\ref*{lem:graphhamcycles}}/n)$-vertex-spread distribution on Hamilton cycles of $G$ (that is, on embeddings $\varphi : C_{n,2,1} \hookrightarrow G)$. 
\end{lemma}
There are elementary approaches to this problem, but here we will just give a short derivation of the lemma from \cite[Lemma 7.3]{PSSS22}, which guarantees $O(1/n)$-vertex-spread distributions on Hamilton paths under the same hypotheses. The task is to just modify the distribution so that it gives Hamilton cycles instead.
\begin{proof}[Proof of Lemma~\ref{lem:graphhamcycles}]
Set $C'=C(2, 1/4)$ from \cite[Lemma 7.3]{PSSS22}. We claim that $C_{\ref*{lem:graphhamcycles}}=10C'$ has the property in the lemma. First note that by \cite[Lemma 7.3]{PSSS22}, there exists a $(C'/n)$-vertex-spread distribution $\mu$ on Hamilton paths of $G$ (that is, on embeddings $\varphi : P_n\hookrightarrow G$, where $P_n$ is $C_{n,2,1}$ with the edge $\{1,n\}$ removed). Let $\phi$ be a random embedding sampled according to $\mu$, and let $\mu'$ be the resulting distribution when $\mu$ is conditioned on the embedding $\phi$ mapping the end points of $P_n$ to adjacent vertices of $G$, i.e. $\{\phi(1),\phi(n)\} \in E(G)$. In the distribution given by $\mu$, for every fixed $u, v \in V(G)$, the probability that $\phi(1) = u$ and $\phi(n) = v$ is at most $(C'/n)^2$ by the vertex-spreadness of $\mu$. Furthermore, by a union bound over all non-adjacent $u,v\in V(G)$, we have that $\mu\left(\{\phi(1), \phi(n)\}\notin E(G)\right)\leq (C'/n)^2 (1/C_{\ref*{lem:graphhamcycles}})n^2 \leq 1/10$. So $\mu'$ is obtained from $\mu$ by conditioning on an event with probability at least $9/10$. Therefore, $\mu'$ is $(2C'/n)$-vertex-spread, so also $(C_{\ref*{lem:graphhamcycles}}/n)$-vertex-spread.
\end{proof}


\section{Random clustering lemma}\label{sec:randomcluster}

This section is devoted to proving the following ``random clustering lemma''.

\begin{lemma}\label{lem:randompartitioninglemma} Let $1 / n \ll 1 / C' \ll 1/C\ll \alpha, \eps, 1/k, 1/d,1/t\leq 1$ where $n,k,d,t,C\in\mathbb N$, and let $\delta \in [0,1]$.  If $H$ is a $k$-uniform hypergraph on $n$ vertices with $\delta_d(H)\geq (\delta +\alpha)\binom{n - d}{k-d}$, then there exists a random partition $\mathcal{U}\coloneqq\{U_1,U_2,\ldots, U_m\}$ of $V(H)$ with the following properties:
\begin{enumerate}[label = {(\arabic{enumi}})]
    \item\label{size} $|U_i|=C$ for each $2\leq i\leq m$ and $|U_1|$ equals $(C - 1)C$ plus the remainder when $n$ is divided by $(C-1)C$;
    \item\label{mindegree} $\delta_d(H[U_i])\geq (\delta +\alpha/2)\binom{|U_i|}{k-d}$ for each $i\in [m]$;
    \item\label{minconnectdegree} for every $i \in [m]$, there exist $u_i \in U_i$, $T_i \subseteq U_i \setminus \{u_i\}$ of size $t$, and $N_i \subseteq \mathcal U\setminus\{U_i\}$ of size at least $(1 - \eps)m$ such that for every $U_j \in N_i$,
    \begin{equation*}
        \delta_d(H[U_i \cup T_j \setminus \{u_i\}]) \geq (\delta + \alpha/2)\binom{|U_i \cup T_j|}{k - d} \qquad \text{and} \qquad\delta_d(H[U_j\cup T_i\setminus \{u_j\}]) \geq (\delta + \alpha / 2)\binom{|U_j\cup T_i|}{k - d};
    \end{equation*}
    \item\label{spreadness} for every set of distinct vertices $y_1, \ldots, y_s\in V(H)$ and every function $f\colon [s]\to [m]$,
    \begin{equation*}
        \Prob{y_i\in U_{f(i)} \text{ for each }i\in [s]}\leq \left(\frac{C'}{n}\right)^s.
    \end{equation*} 
\end{enumerate}
\end{lemma}

For the remainder of this section, fix $1 / n \ll 1 / C' \ll 1/C\ll \alpha, 1/k, 1/d,1/t\leq 1$ where $n,k,d,t,C\in\mathbb N$ as in Lemma~\ref{lem:randompartitioninglemma}.  Let $r$ be $(C - 1)C$ plus the remainder when $n$ is divided by $(C-1)C$. Let $W_1 \coloneqq [r]$, and for $i\geq 2$, let $W_i \coloneqq [r + (C - 1)(i - 1)] \setminus \bigcup_{j=1}^{i-1} W_j=[r + (C - 1)(i - 2)+1, r + (C - 1)(i - 1)]$.  Note that $\{W_1, \dots, W_m\}$ is a partition of $[n]$ for $m\coloneqq (n - r)/(C - 1) + 1$.
Let $H$ be a $k$-uniform hypergraph with vertex set $V \coloneqq \{v_1, \dots, v_n\}$ satisfying $\delta_d(H)\geq (\delta +\alpha)\binom{n - d}{k-d}$.  Let $\pi : [n]\rightarrow [n]$ be a uniformly random permutation of $[n]$, and let $V_i \coloneqq \{v_j : \pi(j) \in W_i\}$.  

As described in Section~\ref{sect:overview}, the partition of $V(H)$ into $V_1, \dots, V_m$ is in a sense ``close'' to the one we want in Lemma~\ref{lem:randompartitioninglemma} because the fourth condition holds, and the second and third condition hold apart from a few exceptions.  Lemmas~\ref{lemma:most-clusters-min-degree}, \ref{lemma:most-clusters-many-vtcs}, and \ref{lemma:V1-in-degree} set up precisely the conditions we need for our \textit{random redistribution} argument to succeed.  Each of these lemmas rely on the following lemma.  

\begin{lemma}\label{lemma:cluster-min-degree-whp}
    For every $T \subseteq V$ of size at most $t$ and every $i \in [m]$,
    \begin{equation*}
    \Prob{\delta_d(H[V_i\cup T])\geq (\delta+2\alpha/3)\binom{|V_i\cup T|}{k-d}}\geq 1-\exp(-C\alpha^2/200).
    \end{equation*} 
    Moreover, the same inequality holds for random sets $T \coloneqq \{v_j : \pi(j) \in W\}$ for every $W\subseteq [n]$ of size at most $t$.
\end{lemma}
\begin{proof}
    Let $T \subseteq V$ of size at most $t$.  Note that each $V_i$ has the same distribution as a uniformly random subset of $V$ of size $C-1$ for every $i \in \{2, \dots, m\}$, and similarly, $V_1$ has the same distribution as a uniformly random subset of $V$ of size $r$. 
    Moreover, for every $W \subseteq [n]$ and $D \subseteq V$ with $|W| = |D|$, in the distribution conditional on the event that $D = \{v_j : \pi(j) \in W\}$, each set $V_i \setminus D$ has the same distribution as a uniformly random subset of $V$ of size $|W_i \setminus W|$.
    
    For every $D' \subseteq T$ of size at most $d$ and $W' \subseteq W_i$ of size $d - |D'|$, let $E_{D',W'}$ be the event that $|D| = d$
    and $\mathrm{deg}(D ,V_i) < (\delta+3\alpha/4)\binom{|V_i|}{k-d}$, where $D \coloneqq \{v_j : \pi(j) \in W'\}\cup D'$.  Note that if $\overline{E_{D', W'}}$ holds for every such $D'$ and $W'$, then $\delta_d(H[V_i\cup T])\geq (\delta+3\alpha/4)\binom{|V_i|}{k-d}\geq (\delta+2\alpha/3)\binom{|V_i\cup T|}{k-d}$.
    We claim that for every such $D'$ and $W'$, we have 
    \begin{equation*}
        \Prob{E_{D', W'}} \leq \exp(-C\alpha^2 / 100).
    \end{equation*}
    Indeed, by the law of total probability, it suffices to show
    \begin{equation*}
        \ProbCond{E_{D', W'}}{\{v_j : \pi(j) \in W'\} = D''} \leq \exp(-C\alpha^2 / 100)
    \end{equation*}
    for every $D'' \subseteq V \setminus D'$ of size $d - |D'|$, which follows from Lemma~\ref{lem:McDAppl} with $D' \cup D''$ playing the role of $D$ and $V_i \setminus D''$ playing the role of $A$.
    Since there are $\binom{|T| + |W_i|}{d} \leq (3C^2)^d$ (using $1/C \ll 1/t$ and $|W_i| \leq 2C^2$) choices for $D'$ and $W'$, the result follows by the union bound, since $1/C\ll \alpha,1/d$.

If instead $T \coloneqq \{v_j : \pi(j) \in W\}$ for some $W \subseteq [n]$ of size at most $t$, then the proof is essentially the same, so we omit it.
\end{proof}

The next lemma will be used to ensure \ref{minconnectdegree} holds.  When applied with $T = \{v\}$ for a vertex $v$ in a ``bad'' cluster, it also ensures that there are many good clusters to which $v$ can be added in the random redistribution step.
\begin{lemma}\label{lemma:most-clusters-min-degree}
    With probability at least $99/100$, for every $T \subseteq V$ of size at most $t$, there are at least $(1 - 1/C^2)m$ sets $V_i$ for which $\delta_d(H[V_i\cup T])\geq  (\delta+2\alpha/3)\binom{|V_i\cup T|}{k-d}$ and for every $W \subseteq [n]$ of size at most $t$, the same property holds for the random set $T \coloneqq \{v_j : \pi(j) \in W\}$.
\end{lemma}
\begin{proof} 
For every $T \subseteq V$ of size at most $t$, let $X_T$ denote the number of random sets $V_i$ such that the minimum degree condition above holds. By linearity of expectation and Lemma~\ref{lemma:cluster-min-degree-whp}, we have $\Expect{X_T} \geq (1 - \exp(-C\alpha^2/200))m$.  
 Note that interchanging two elements of the random permutation $\pi$ either has no effect on $V_1, \dots, V_m$ or changes the values of two random sets $V_i$ and $V_j$ and therefore can affect the value of $X_T$ by at most $2$. Also, if $X_T\geq s$, this can be certified by at most $2C^2s$ choices of the random permutation. Therefore, by Lemma~\ref{lem:mcdiarmidperm} applied with $c=2$, $r=2C^2$, and $t = m / C^3$, we have
\begin{equation*}
    \Prob{X_T < (1 - 1/C^2)m} \leq 4\exp\left(-\frac{(m/C^3)^2}{64C^2m}\right) \leq \exp(-n / C^{10}).
\end{equation*}
By a union bound, we have that $X_T \geq (1 - 1/C^2)m$ for every $T \subseteq V$ of size at most $t$ with probability at least $99/100$, as desired.

If instead $T \coloneqq \{v_j : \pi(j) \in W\}$ for some $W \subseteq [n]$ of size at most $t$, then the proof is essentially the same, so we omit it.
\end{proof}

The next lemma ensures that most clusters are good.  Moreover, in our random redistribution argument, we will assign every good cluster a vertex from a bad cluster.  The lemma also ensures that for every good cluster (except $V_1$), there are many options for this vertex.

\begin{lemma}\label{lemma:most-clusters-many-vtcs}
    With probability at least $99/100$, for all but at most $\exp\left(-\alpha^2C/1000\right)m$  many $i\in \{2,3,\ldots, m\}$, there are at least $(1-\exp\left(-\alpha^2C/500\right))n$ vertices $v\in V$ such that $\delta_d(H[V_i\cup \{v\}])\geq  (\delta+2\alpha/3)\binom{|V_i\cup \{v\}|}{k-d}$.
\end{lemma}
\begin{proof} 
    By Lemma~\ref{lemma:cluster-min-degree-whp}, we have that $\Prob{\delta_d(H[V_i\cup \{v\}])<  (\delta+2\alpha/3)\binom{|V_i\cup \{v\}|}{k-d}}\leq \exp\left(-\alpha^2C/200\right)$ for any $v\in V$ and any random set $V_i$. Call $v$ \textit{bad} for $V_i$ if this low probability event holds. So for any random set $V_i$, we have $\Expect{|\{v\in V\colon v \text { is bad for } V_i\}}]\leq \exp\left(-\alpha^2C/200\right)n$ by linearity of expectation. By Markov's inequality, for any random set $V_i$, we have $\Prob{|\{v\in V\colon v \text { is bad for } V_i\}|\geq \exp\left(-\alpha^2C/500\right)n} \leq \exp\left(-\alpha^2C/500\right)$. Let $X$ denote the number of $V_i$ such that $|\{v\in V\colon v \text { is bad for } V_i\}|\geq \exp\left(-\alpha^2C/500\right)n$.  By linearity of expectation, $\Expect{X}\leq \exp\left(-\alpha^2C/500\right)m$. By Markov's inequality, $\Prob{X\geq \exp\left(-C\alpha^2/1000\right)m}\leq \exp\left(-C\alpha^2/1000\right)\leq 1/100$, implying the desired statement.
\end{proof}

For every $i \in [m]$, choose some $W'_i \subseteq W_i$ of size $t$ arbitrarily (irrespective of $\pi$), let $T_i \coloneqq \{v_j : \pi(j) \in W'_i\}$, let $N^+_i$ be the set of $V_j \in \{V_1, \dots, V_m\}$ for $j \neq i$ such that $\delta_d(H[V_j \cup T_i]) \geq (\delta + 2\alpha/3)\binom{|V_j \cup T_i|}{k - d}$, and let $N^-_i$ be the set of $V_j$ for $j \neq i$ such that $V_i \in N^+_j$.  Note that $T_i \subseteq V_i$ for all $i \in [m]$.  In order to ensure \ref{size} holds, we will need to have $U_1 = V_1$.  The next lemma ensures that $U_1$ (which will be $V_1)$ satisfies \ref{mindegree} and \ref{minconnectdegree}.
\begin{lemma}\label{lemma:V1-in-degree}
    With probability at least $99/100$, $|N^-_1| \geq (1 - \exp(-\alpha^2C / 500))(m - 1)$.
\end{lemma}
\begin{proof}
    By the ``moreover'' part of Lemma~\ref{lemma:cluster-min-degree-whp}, we have that $$\Prob{\delta_d(H[V_1\cup T_i])<  (\delta+2\alpha/3)\binom{|V_1\cup T_i|}{k-d}}\leq \exp\left(-\alpha^2C/200\right)$$ for every $i \geq 2$, so $$\Prob{V_i \in N^-_1} \geq 1 - \exp\left(-\alpha^2C/200\right).$$ By linearity of expectation, $\Expect{|N^-_1|} \geq (1 - \exp\left(-\alpha^2C/200\right))(m - 1)$.  By Markov's inequality, $$\Prob{m - 1 - |N^-_1| \geq \exp\left(-C\alpha^2/500\right)(m - 1)}\leq \exp\left(-C\alpha^2/500\right)\leq 1/100,$$ implying the desired statement.
\end{proof}

Now we prove Lemma~\ref{lem:randompartitioninglemma}, by considering the distribution on $V_1, \dots, V_m$ conditional on the events in the previous three lemmas and applying the random redistribution argument.

\begin{proof}[Proof of Lemma~\ref{lem:randompartitioninglemma}] 
Let $E_1$ be the event that the property in Lemma~\ref{lemma:most-clusters-min-degree} holds,
let $E_2$ be the event that the property in Lemma~\ref{lemma:most-clusters-many-vtcs} holds,
and let $E_3$ be the event that the property in Lemma~\ref{lemma:V1-in-degree} holds.
We will condition on $E_1 \cap E_2 \cap E_3$, which holds with probability at least $97 / 100$.

For each permutation $\pi : [n] \rightarrow [n]$, define a set of ``bad'' random sets $\mathcal{F}_\pi\subseteq \{V_2,\ldots, V_m\}$ to include the $V_i$ for which we have we have any of 
\begin{enumerate}[label=(A\arabic*)]
    \item\label{bad-set:low-degree} $\delta_d(H[V_i])< (\delta+2\alpha/3)\binom{|V_i|}{k-d}$, 
    \item\label{bad-set:bipartite-degree} $\delta_d(H[V_i\cup\{v\}])<(\delta+2\alpha/3)\binom{|V_i\cup\{v\}|}{k-d}$ for at least $\exp(-\alpha^2C / 500)n$ vertices $v \in V$, and
    \item\label{bad-set:in-degree} $|N^-_i| < (1 - 1/\sqrt{C})m$. 
\end{enumerate}

Let $m'\coloneqq |V_2 \cup \cdots \cup V_m|/C$, and note that $m'$ is a positive integer by the choice of $r$.  If necessary, add extra elements of $\{V_2, \dots, V_m\}$ arbitrarily to $\cF_\pi$ to ensure $\cF_\pi$ has size at least $m-1-m' = m - 1 - m(C - 1)/C = m/C - 1$.  We claim that if $\pi \in E_1 \cap E_2 \cap E_3$, then $|\cF_\pi| = m - 1 - m'$.  Indeed, there are at most $m / C^2$ of type \ref{bad-set:low-degree} by Lemma~\ref{lemma:most-clusters-min-degree} applied with $T=\emptyset$ assuming $\pi \in E_1$, there are at most $\exp\left(-\alpha^2C/1000\right)m$ of type \ref{bad-set:bipartite-degree} by Lemma~\ref{lemma:most-clusters-many-vtcs} assuming $\pi \in E_2$, and there are at most $m / C^{4/3}$ of type \ref{bad-set:in-degree} since $\sum_{i=1}^m|N^-_i| = \sum_{i=1}^m|N^+_i| \geq (1 - 1/C^2)m^2$ by Lemma~\ref{lemma:most-clusters-min-degree} assuming $\pi \in E_1$.

By possibly relabelling the sets $V_1, \dots, V_m$, we assume without loss of generality that $\cF_\pi = \{V_i : i \in [m]\setminus [m +1 - |\cF_\pi|]\}$.  
Now consider random sets $V_i$ given by a random permutation $\pi$ conditioned on $E_1\cap E_2\cap E_3$.

We define a bipartite graph $G_\pi$ between $\bigcup_{V_i \in\cF_\pi}V_i$ and $\{V_2,\ldots, V_m\}\setminus \cF_\pi$. Note that the first part is a set of vertices and the second part is a set of subsets of vertices, and by the choice of $m'$, the first and the second part both have size $m'$, as $(m-1-m')(C-1)=m'$. We put an edge between $v$ and $V_i$ whenever $\delta_d(H[V_i\cup \{v\}])\geq  (\delta+2\alpha/3)\binom{|V_i\cup \{v\}|}{k-d}$. Conditioning on $E_1$ and considering $T = \{v\}$, we have $d_G(v) \geq (1 - 1/C^2)m - |\cF_\pi| \geq 3m' / 4$ for every $v \in \bigcup_{V_i\in \cF_\pi}V_i$.  By the choice of $\cF_\pi$ to contain all sets satisfying \ref{bad-set:bipartite-degree}, for every $V_i \in \{V_2,\ldots, V_m\}\setminus \cF_\pi$, we have $d_G(V_i) \geq m' - \exp(-\alpha^2 C / 500)n \geq 3m' / 4$.  Thus, $\delta(G)\geq 3m'/4$, and by Lemma~\ref{lem:graphmatchings}, there is a $(C_{\ref*{lem:graphmatchings}}/m)$-spread distribution on perfect matchings of $G_\pi$. 

We define the random sets $U_i$ and vertices $u_i$ as follows. First, sample $\pi$ from the uniform distribution on permutations of $[n]$ conditional on $E_1 \cap E_2 \cap E_3$.  Then, sample $M_\pi$ from the $(C_{\ref*{lem:graphmatchings}}/m)$-spread distribution on perfect matchings of $G_\pi$.  For each $V_i \notin \cF_\pi$, let $u_i$ be the vertex that it is assigned to in $M_\pi$, and let $U_i = V_i \cup \{u_i\}$. Letting $U_1 = V_1$, and picking $u_1$ to be an arbitrary vertex in $V_1\setminus T_1$  concludes the algorithm that defines $\mathcal{U} = \{U_1, \dots, U_{m' + 1}\}$ with $m$ in the lemma statement being $m'+1$.  We now check the required conditions.
\begin{enumerate}[label = {(\arabic{enumi}})]
\item The condition on $|U_1|$ holds by construction, and for $i\geq 2$, we have $|U_i| = |V_i| + 1 = C$, as required.
\item For $i=1$, this condition holds because we condition on $E_3$, and for $i\geq 2$, the condition holds by the definition of the edges in the bipartite graph $G_\pi$. 
\item For every $i \in [m]$, let $N_i = \{U_j : V_j \in N^+_i \cap N^-_i\}$. Note that $|N^+_i|\geq (1-1/C^2)m$ since we are conditioning on $E_1$. Note that $|N_1^-|\geq (1-\exp(-\alpha^2C/500))(m-1)$ since we are conditioning on $E_3$.  Note that for $i\in[2, m'+1]$, we have $|N_i^-|\geq (1-1/\sqrt C)m$ since these $V_i$ are good sets and so satisfy the reverse of (3). Combining these gives $|N_i| \geq m' - m / C^2 - m/\sqrt C \geq (1 - \eps)(m' + 1)$. 
For $i\geq 2$, note that  $\delta_d(H[U_i \cup T_j\setminus\{u_i\}]) = \delta_d(H[V_i \cup T_j])$. Also note $\delta_d(H[U_1 \cup T_j\setminus\{u_1\}]) =  \delta_d(H[V_1 \cup T_j]\setminus \{u_1\})\geq \delta_d(H[V_1 \cup T_j])-\binom{|V_1 \cup T_j| - 1}{k - d - 1}\geq \delta_d(H[V_1 \cup T_j]) -  \alpha\binom{|V_1 \cup T_j|}{k - d}/10$. Set $s_{i,j}:=\alpha\binom{|V_i \cup T_j|}{k - d}/10$. We thus have 
\begin{equation*}
    \delta_d(H[U_i \cup T_j\setminus\{u_i\}]) \geq  \delta_d(H[V_i \cup T_j]) - s_{i,j}\geq (\delta + 2\alpha / 3)\binom{|V_i \cup T_j|}{k - d}-s_{i,j} \geq (\delta + \alpha / 2)\binom{|U_i \cup T_j|}{k - d}
\end{equation*}
for $i \geq 1$ and $U_j \in N_i$ because $V_j \in N^+_i$, and similarly
\begin{equation*}
    \delta_d(H[U_j \cup T_i\setminus\{u_j\}]) \geq \delta_d(H[V_j \cup T_i]) -s_{i,j}\geq (\delta + 2\alpha / 3)\binom{|V_j \cup T_i|}{k - d}-s_{i,j} \geq (\delta + \alpha / 2)\binom{|U_j \cup T_i|}{k - d}
\end{equation*}
for $i \geq 1$ and $U_j \in N_i$ because $V_j \in N^-_i$.

\item Let $y_1,\ldots, y_s\in V(H)$, and let $f : [s] \rightarrow [m' + 1]$ as in the statement.  For each $i \in [s]$, let $D_i$ be the event that $y_i \in U_{f(i)}$, let $D_i^1$ be the event $y_i \in V_{f(i)}$, and let $D_i^2$ be the event $y_i = u_i$.  Note that $D_i^1$ is determined only by the permutation $\pi$, and the event $D_i^2$ only holds if $y_i$ is matched to $V_i$ in $G_\pi$ by $M_\pi$.  Thus, for every $S \subseteq [s]$, we have

\begin{equation*}
    \Prob{\bigcap_{i \in S}D^1_i} \leq \left.\left(\frac{C^{2|S|}(n-|S|)!}{n!}\right)\middle/\Prob{{E_1 \cap E_2 \cap E_3}}\right. \leq \frac{100}{97}\left(\frac{eC^2}{n}\right)^{|S|},
\end{equation*}

and for every $\pi' \in E_1 \cap E_2 \cap E_3$,

\begin{equation*}
    \ProbCond{\bigcap_{i \in [s]\setminus S} D_i^2}{\pi = \pi'} \leq \left(\frac{C_{\ref*{lem:graphmatchings}}}{m'}\right)^{s - |S|}.
\end{equation*}
Therefore, for every $S \subseteq [s]$, we have
\begin{equation*}
    \Prob{\bigcap_{i \in S}D_i^1 \cap \bigcap_{i \in [s]\setminus S} D_i^2} = \Prob{\bigcap_{i \in S}D^1_i}\ProbCond{\bigcap_{i \in [s]\setminus [S]} D_i^2}{\bigcap_{i \in S}D_i^1} \leq \frac{100}{97}\left(\frac{eC^2}{n}\right)^{|S|}\left(\frac{C_{\ref*{lem:graphmatchings}}}{m'}\right)^{s - |S|} \leq \left(\frac{C'}{2n}\right)^s.
\end{equation*}
Finally, the result follows by a union bound over the $2^s$ choices of $S \subseteq [s]$, since 
\begin{equation*}
    \bigcap_{i \in [s]}D_i = \bigcap_{i \in [s]}\left(D_i^1 \cup D_i^2\right) = \bigcup_{S \subseteq [s]}\left(\bigcap_{i \in S}D_i^1 \cap \bigcap_{i \in [s]\setminus S} D_i^2 \right).\pushQED{\qed} \qedhere
\end{equation*}

\end{enumerate}
\end{proof}
\section{Proof of Theorems~\ref{thm:main-vertex-spread} and \ref{thm:vtx-spread-for-factors}}

In this section, we prove Theorems~\ref{thm:main-vertex-spread} and \ref{thm:vtx-spread-for-factors}.  Since the proof of Theorem~\ref{thm:main-vertex-spread} is the most challenging, we do not include all of the details for the proof of Theorem~\ref{thm:vtx-spread-for-factors} and instead describe at the end of this section how to modify the proof of Theorem~\ref{thm:main-vertex-spread} to prove Theorem~\ref{thm:vtx-spread-for-factors}.

Recall that a $k$-uniform Hamilton $\ell$-path with $s$ edges occupies $(s-1)(k-\ell) + k= s(k-\ell)+\ell$ vertices. We call a set $S$ $(k,\ell)$-\textit{path-divisible} if $|S|$ has the right divisibility condition to contain a $k$-uniform Hamilton $\ell$-path, that is, $k-\ell$ divides $|S|-\ell$. We call a set $S$ $(k,\ell)$-\textit{internal-path-divisible} if $|S|+2\ell$ has the right divisibility condition to contain a $k$-uniform Hamilton $\ell$-path, so $k-\ell$ divides $|S|+\ell$. We call an integer $q$ path-divisible (or internal-path-divisible) if a set of size $q$ is path-divisible (or internal-path-divisible). Let $f(k, \ell) \coloneqq \lceil (2k - \ell)/(k - \ell)\rceil$, and note that the final $\ell$ vertices of an $\ell$ path with $f(k, \ell)$ edges are disjoint from the first $\ell$.

We now describe the random Hamilton $\ell$-cycle embedding algorithm.
In this algorithm, we first find a Hamilton cycle in $H$ by gluing together many $\ell$-paths.  
We assume each of these paths comes with a natural ordering of its vertices.
For technical reasons, it is more convenient to find an embedding of $C'_{n,k,\ell}$ into $H$, where $C'_{n,k,\ell}$ is the hypergraph isomorphic to $C_{n,k,\ell}$ where edges are shifted to the left by $\ell$; that is, $C'_{n,k,\ell}$ has vertex set $[n]$, $\{n-\ell+1, \ldots, n, 1,\ldots, k-\ell\}\in E(C'_{n,k,\ell})$, and all other edges are ``$(k-\ell)$-shifts'' of this edge.
See Figure~\ref{fig:embedding-algorithm} for an illustration of the embedding algorithm.


\begin{definbox}[Random Hamilton cycle embedding algorithm]\label{def:distribution}
Let $\alpha>0$, let $k \in \mathbb N$, and let $d,\ell \in [k - 1]$.
Let $1 / n \ll 1 / C' \ll 1 / C \ll \alpha, 1/k$ as in Lemma~\ref{lem:randompartitioninglemma}, where $k-\ell$ divides $n$, and $C-(k-\ell)f(k,\ell)-\ell$ is $(k,\ell)$-internal-path divisible, meaning that $k-\ell$ divides $C$.
Given a hypergraph $H$ on $n$ vertices with $\delta_d(H)\geq (\hamconnthresh_{k,\ell, d}+\alpha)\binom{n - d}{k-d}$, we define $\psi : C'_{n,k,\ell} \hookrightarrow H$, a random embedding of a Hamilton $\ell$-cycle in $H$, as follows.

\textbf{Step 1: Sample random clusters.} 
Let $\mathcal U = \{U_1,\ldots, U_m\}$ be a random partition of $V(H)$ obtained by applying Lemma~\ref{lem:randompartitioninglemma} with $10\ell k$ playing the role of $t$, $1/5$ playing the role of $\eps$, and $\hamconnthresh_{k,\ell, d}$ playing the role of $\delta$ (the choice of the other variables is as above).  Define an auxiliary graph $G$ with vertex set $[m]$ where $i$ is adjacent to $j$ if $U_j \in N_i$ or $U_i \in N_j$.  Note that by the choice of $\eps$, we have $\delta(G) \geq 4m/5$.

Let $\phi : C_{m,2,1} \hookrightarrow G$ be a random embedding of a Hamilton cycle on $G$ coming from Lemma~\ref{lem:graphhamcycles}, and note that $\phi$ is a permutation of $[m]$. 
 Throughout, we view $[m]$ cyclically, meaning occurrences of the indices $m+1$ and $0$ are to be read as $1$ and $m$, respectively.
Let $z = \phi^{-1}(1)$ so that $U_1=U_{\phi(z)}$, the largest random set. 

\textbf{Step 2: Find connecting paths between clusters.} 
For each $i\in [m]$, find an $\ell$-path $P_{\phi(i), \phi(i+1)}$ of length $f(k,\ell)$ contained in $U_{\phi(i)}\cup U_{\phi(i+1)}$ with the following properties.
\begin{enumerate}
    \item The paths $P_{\phi(i), \phi(i+1)}$ are pairwise vertex-disjoint.
    \item $P_{\phi(i), \phi(i+1)}$ has its first $\ell$ vertices in $U_{\phi(i)}$, and the last $|V(P_{\phi(i), \phi(i+1)})|-\ell\geq \ell$ vertices in $U_{\phi(i+1)}$ (that is, $P_{\phi(i), \phi(i+1)}$ does not alternate between clusters). 
\end{enumerate}

\textbf{Step 3: Find Hamilton paths in the leftover part of each cluster.} For each $i\in [m]$, call the $\ell$-set that is an endpoint of $P_{\phi(i-1), \phi(i)}$ in $U_{\phi(i)}$ as $S$ and call the $\ell$-set that is an endpoint of $P_{\phi(i), \phi(i+1)}$ in $U_{\phi(i)}$ as $T$. Find a Hamilton $\ell$-path in $H[U_{\phi(i)}\setminus \left(V(P_{\phi(i),\phi(i+1)})\cup V(P_{\phi(i-1),\phi(i)})\right) \cup S \cup T]$ with endpoints $S$ and $T$, and denote this Hamilton path as $P_{\phi(i)}$.

\textbf{Defining $\psi$.} 
Recall that $z = \phi^{-1}(1)$.  To define $\psi : [n] \rightarrow V(H)$, we define $\psi(1)$ to be the first vertex of $P_{\phi(z-1), \phi(z)}$ that is contained in $U_{\phi(z)}$ and define the remaining $\psi(i)$ for $i > 1$ to be consistent with the ordering given by $P_{\phi(1)} \cup P_{\phi(1), \phi(2)} \cup P_{\phi(2)} \cup \cdots \cup  P_{\phi(m)} \cup P_{\phi(m), \phi(1)}$.
\end{definbox}

\begin{figure}
    \centering
    \begin{tikzpicture}[scale=\textwidth/30cm]
        \draw (0,0) ellipse (-12cm and 1cm);
        \node[label=center:{$H$}] at (0, 0) {};

        \draw[->] (0, 1.5) -- (0, 2.5);
        \node[label=right:{Lemma~\ref{lem:randompartitioninglemma}}] at (0, 2) {};
        
        \node (U1) at (-9.5, 5) {};
        \node (U2) at (-5, 5) {};
        \node (Um) at (10, 5) {};
        \draw (U1) circle (2cm);
        \draw (U2) circle (1.5cm);
        \draw (Um) circle (1.5cm);
        \node[label=center:{$U_1$}] at (U1) {};
        \node[label=center:{$U_2$}] at (U2) {};
        \node[label=center:{$U_m$}] at (Um) {};
        \node[label=center:{\Huge$\cdots$}] at ($(U2)!.5!(Um)$) {};

        \node (U'1) at (-9.5, 11) {};
        \node (U'2) at (-5, 11) {};
        \node (U'm) at (10, 11) {};
        \draw (U'1) circle (2cm);
        \draw (U'2) circle (1.5cm);
        \draw (U'm) circle (1.5cm);
        \node[label=center:{$U_{\phi(z)}$}] at (U'1) {};
        \node[label=center:{$U_{\phi(z + 1)}$}] at (U'2) {};
        \node[label=center:{$U_{\phi(z - 1)}$}] at (U'm) {};
        \node[label=center:{\Huge$\cdots$}] at ($(U'2)!.5!(U'm)$) {};

        \draw[->] (-9.5, 7.5) -- (-9.5, 8.5);
        
        \draw[->] (2.5, 7) -- (2.5, 9);
        \draw[->] (0, 7) .. controls (0, 8.5) and (5, 7.5) .. (5, 9);
        \draw[->] (5, 7) .. controls (5, 8.5) and (0, 7.5) .. (0, 9);
        \draw[->] (1.25, 7) .. controls (1.25, 8.5) and (3.75, 7.5) .. (3.75, 9);
        \draw[->] (3.75, 7) .. controls (3.75, 8.5) and (1.25, 7.5) .. (1.25, 9);
        
        \draw (-12, 15) -- (12, 15) to[bend right=7] (-12, 15);
        \node[label=center:{$C'_{n,k,\ell}$}] at (0, 15.4) {};
        \draw[->] (.5, 14) arc[start angle=0, end angle=180,radius=.25cm] -- (0, 13);
        \node[label=right:{$\psi$ (Definition~\ref{def:distribution})}] at (0, 13.5) {};

        \foreach \x in {-12,-7,-3,8,12} \draw(\x,14.9)--(\x,15.1); 

        \draw[decorate, decoration={brace,mirror}, yshift=-2.5ex]  (-11.9, 15) -- node[below=0.5ex] {$W_1$}  (-7.1,15);
        \draw[decorate, decoration={brace,mirror}, yshift=-2.5ex]  (-6.9, 15) -- node[below=0.5ex] {$W_2$}  (-3.1,15);
        \draw[decorate, decoration={brace,mirror}, yshift=-2.5ex]  (8.1, 15) -- node[below=0.5ex] {$W_m$}  (11.9,15);
    \end{tikzpicture}
    \caption{An illustration of the random Hamilton cycle embedding algorithm}
    \label{fig:embedding-algorithm}
\end{figure}
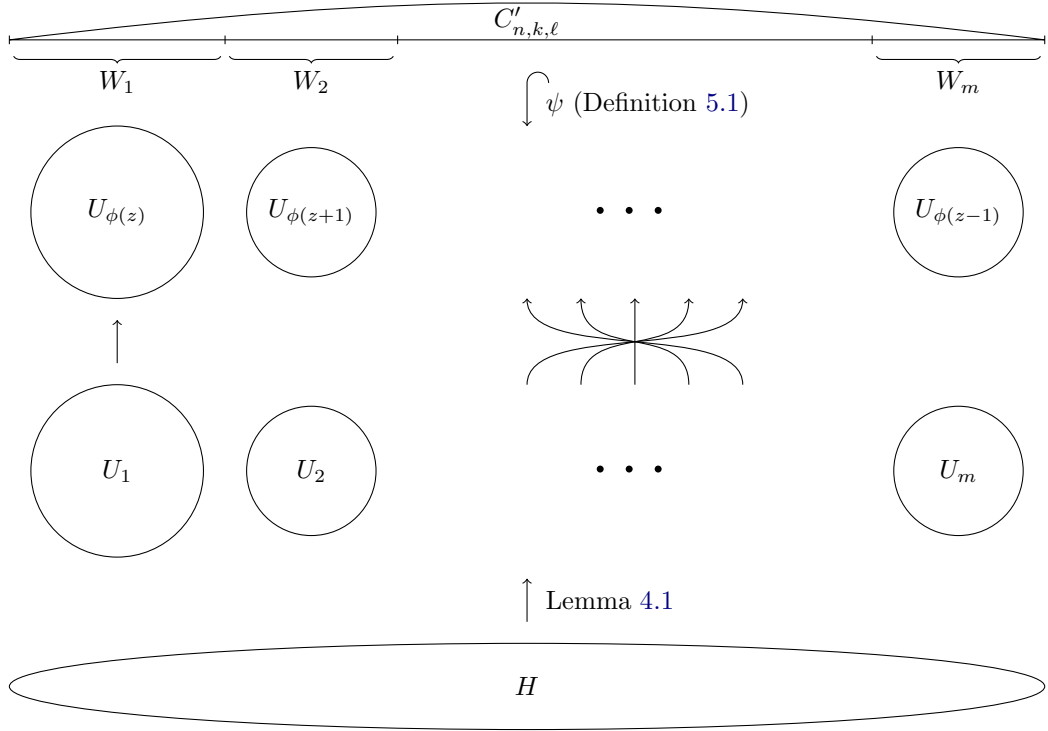

For the next two lemmas, fix $\alpha>0$, $k \in \mathbb N$, $d,\ell \in [k - 1]$, and $1 / n \ll 1 / C' \ll 1 / C \ll \alpha, 1$ as in Definition~\ref{def:distribution}, and let $H$ be an $n$-vertex hypergraph with $\delta_d(H) \geq (\hamconnthresh_{k,\ell,d} + \alpha)\binom{n - d}{k - d}$.
Theorem~\ref{thm:main-vertex-spread} follows immediately from the following two lemmas (note that producing embeddings of $C'_{n,k,\ell}$ is equivalent to producing those of $C_{n,k,\ell}$).
\begin{lemma}\label{lemma:distribution-well-defined}
    The random embedding $\psi : C'_{n,k,\ell}\hookrightarrow H$ in Definition~\ref{def:distribution} is well-defined.
\end{lemma}

\begin{lemma}\label{lemma:distribution-is-spread}
    The distribution produced by the algorithm in Definition~\ref{def:distribution} is $(C' C_{\ref*{lem:graphhamcycles}}/n)$-vertex-spread.
\end{lemma}

\begin{proof}[Proof of Lemma~\ref{lemma:distribution-well-defined}]
By construction, $P_{\phi(1)} \cup P_{\phi(1), \phi(2)} \cup P_{\phi(2)} \cup \cdots \cup  P_{\phi(m)} \cup P_{\phi(m), \phi(1)}$ is a Hamilton $\ell$-cycle of $H$. Furthermore, by Step 2, the first vertex of $P_{\phi(z-1), \phi(z)}$ that is contained in $U_{\phi(z)}$ is in fact the $(\ell+1)$th vertex of $P_{\phi(z-1), \phi(z)}$, so 
$\psi$ is an embedding of $C'_{n,k,\ell}$ in $H$, as required.  Therefore it suffices to show that each step in Definition~\ref{def:distribution} can be performed. 

\textbf{Step 1.}  
This step can be performed because of Lemmas~\ref{lem:randompartitioninglemma} and \ref{lem:graphhamcycles}.

\textbf{Step 2.} Let $i\in [m]$ and suppose paths satisfying the two properties have been found for each $i'<i$, and let $J$ denote the set of vertices used by these previously found paths. Consider $U_{\phi(i)}$ and $U_{\phi(i+1)}$. Note that from each set, $J$ uses at most $\ell + (k - \ell)f(k,\ell)\leq 3k$ vertices. Consider the $T_{\phi(i)}\subseteq U_{\phi(i)}$ guaranteed by Lemma~\ref{lem:randompartitioninglemma}\ref{minconnectdegree}. Note that $t-3k\geq \ell$, so we may fix an $\ell$-subset $T'\subseteq T_{\phi(i)}$ disjoint from $J$. Let $U\subseteq U_{\phi(i+1)}\setminus (J\cup\{u_{\phi(i+1)}\})$ be obtained by deleting at most $k$ elements from the latter superset so that $U\cup T'$ is $(k,\ell)$-path divisible.
Observe that $\delta_{d}(H[U\cup T'])\geq (\delta_{k,\ell,d}^{\mathrm{CON}}+\alpha/10)\binom{|U\cup T'|}{k-d}$ (by Lemma~\ref{lem:randompartitioninglemma}\ref{minconnectdegree}); hence, by definition of $\delta_{k,\ell,d}^{\mathrm{CON}}$, $H[U\cup T']$ contains a Hamilton $\ell$-path with one endpoint being the $\ell$-set $T'$, and the other endpoint being some $\ell$-subset of $U$ that can be chosen arbitrarily. 
We truncate this path (keeping the side closer to $T'$ intact) to keep only its first $f(k,\ell)$ edges. The resulting path satisfies the desired properties, showing that Step $2$ can be executed for each $i\in [m]$.
\par \textbf{Step 3.} Similarly to the previous step, we need to check the relevant divisibility and minimum degree conditions to show that a path of the desired form exists for each $i\in [m]$.

We first show that $U_{\phi(i)}\setminus (V(P_{\phi(i),\phi(i+1)})\cup V(P_{\phi(i-1),\phi(i)}))$ is $(k,\ell)$-internal-path-divisible for every $i \in [n]$. Indeed, for $i \neq \phi^{-1}(1)$, we have $|U_{\phi(i)}\setminus (V(P_{\phi(i),\phi(i+1)})\cup V(P_{\phi(i-1),\phi(i)}))| = C - \ell - f(k,\ell)(k - \ell)$, so the internal path divisibility follows from our choice of $C$.  For $i = \phi^{-1}(1)$, we have $|U_{\phi(i)}\setminus (V(P_{\phi(i),\phi(i+1)})\cup V(P_{\phi(i-1),\phi(i)}))| = n - \ell - f(k,\ell)(k - \ell) - (m - 1)C$, so the internal path divisibility follows since $k - \ell$ divides both $n$ and $C$.

Now we claim that $\delta_d(H[U_{\phi(i)}\setminus (V(P_{\phi(i),\phi(i+1)})\cup V(P_{\phi(i-1),\phi(i)}))])\geq (\delta_{k,\ell,d}+\alpha/10)\binom{|U_{\phi(i)}|}{k-d}$. This is a consequence of Lemma~\ref{lem:randompartitioninglemma}\ref{mindegree} since $|V(P_{\phi(i),\phi(i+1)})\cup V(P_{\phi(i-1),\phi(i)})|\leq 10k$, and $1/C\ll 1/k$.  Therefore, the desired Hamilton path with the desired endpoints exists by the definition of $\delta_{k,\ell,d}^{\mathrm{CON}}$, meaning that Step $3$ can be executed.
\end{proof}

\begin{proof}[Proof of Lemma~\ref{lemma:distribution-is-spread}]
Let $\psi : C'_{n,k,\ell} \hookrightarrow H$ be the random embedding from Definition~\ref{def:distribution}.  To show the distribution is $(C'C_{\ref*{lem:graphhamcycles}}/n)$-vertex-spread, we show that for every $s\in [n]$ and every two sequences of distinct $x_1, \dots, x_s  \subseteq [n]$ and $y_1, \dots, y_s \in V(H)$, we have $\Prob{\psi(x_i) = y_i \text{ for all } i\in [s]} \leq (C'C_{\ref*{lem:graphhamcycles}}/n)^s$.

To that end, let $r$ be $(C - 1)C$ plus the remainder when $n$ is divided by $(C-1)C$. Let $W_1 \coloneqq [r]$, and for $i\geq 2$, let $W_i \coloneqq  [r + C(i - 2)+1, r + C(i - 1)] = [r + C(i - 1)] \setminus \bigcup_{j=1}^{i-1} W_j$.  (Note that these are slightly different from those defined in the previous section.)  
We refer to each $W_i$ as a ``window''.   
Let $w : [s] \rightarrow [m]$ where $w(i)$ is the unique index such that $x_i \in W_{w(i)}$. 
We refer to $W_{w(i)}$ as ``$x_i$'s window''.  
Recall $\mathcal{U} = \{U_1, \dots, U_m\}$ is the collection of random sets and $\phi$ is the random embedding of $C_{m,2,1}$ in the auxiliary graph $G$ as defined in Step 1, and $z = \phi^{-1}(1)$.  See Figure~\ref{fig:embedding-algorithm}.  We will frequently use the following.
\begin{observation}\label{obs:windows}
    The embedding $\psi$ induces bijections between and $W_{i+1}$ and $U_{\phi(z+i)}$ for each $i\geq 0$ (where the subscripts of $W$ and the inputs to $\phi$ are to be interpreted cyclically).  In particular, for every $i \in [s]$, we have $\psi(x_i) = y_i$ only if $y_i \in U_{\phi(z + w(i) - 1)}$.
\end{observation}

For each function $f : [s]\rightarrow [m]$, let $E_f$ be the event that $y_i\in U_{f(i)}$ for each $i \in [s]$.
Let $\cF$ be the family of functions $f\colon [s]\to [m]$ satisfying the following:
\begin{itemize}
    \item If $w(i) = 1$ (that is, $x_i$ is in the biggest window), then $f(i) = 1$.
    \item If $w(i) = w(j)$ (that is, $x_i$ and $x_j$ are in the same window), then $f(i) =  f(j)$.
    \item If $w(i) \neq w(j)$ (that is, $x_i$ and $x_j$ are in different windows), then $f(i) \neq f(j)$.
\end{itemize}
Let $b$ denote the number of $i \in [m]$ for which $\{x_1, \dots, x_s\} \cap W_i \neq \emptyset$; that is, $b$ is the number of windows that are the window of some $x_i$ for $i \in [s]$. 
Define $\hat{b}$ as $b$ if $W_1\cap \{x_1, \dots, x_s\}=\emptyset$ and define $\hat{b}$ as $b-1$ otherwise.

We claim the following:
\begin{enumerate}[label=(\alph*)]
    \item\label{eqn:cluster-allocation-probability} 
    $\Prob{E_f} \leq (C' / n)^s$ for every  $f : [s]\rightarrow[m]$;
    \item \label{eqn:compatibility}
    if $f : [s] \rightarrow [m]$ is not in $\cF$, 
        then $\ProbCond{\psi(x_i) = y_i \text{ for all } i \in [s]}{E_f}=0$;
    \item \label{eqn:number-compatible}
    $|\cF| = \prod_{i=0}^{\hat{b}-1}(m - i) \leq m^{\hat{b}}$;
    \item \label{eqn:hat-b-bound}
    $\hat{b}\leq b \leq s$;
    \item \label{claim:boundprob}
    $\ProbCond{\psi(x_i) = y_i \text{ for all } i \in [s]}{E_f}\leq (C_{\ref*{lem:graphhamcycles}}/m)^{\hat b}$ for every $f\in \mathcal{F}$.
\end{enumerate}

Indeed, \ref{eqn:cluster-allocation-probability} follows immediately from Lemma~\ref{lem:randompartitioninglemma}\ref{spreadness}, \ref{eqn:compatibility} follows from Observation~\ref{obs:windows}, and \ref{eqn:number-compatible} and \ref{eqn:hat-b-bound} are straightforward from the definitions.
To prove \ref{claim:boundprob}, first note that by Observation~\ref{obs:windows},
\begin{equation*}
    \ProbCond{\psi(x_i) = y_i \text{ for all } i \in [s]}{E_f}
    \leq \Prob{\phi(z + w(i) - 1) = f(i) \text{ for all } i\in[s]}.
\end{equation*}
        
Let $x'_1, \dots, x'_{\hat b}$ and $y'_1, \dots, y'_{\hat b}$ be the elements of $\{z + w(i) - 1 : i \in [s]\}\setminus \{z\}$ and $\{f(i) : i \in [s]\}\setminus \{1\}$, respectively.  
By the vertex spreadness of $\phi$ with $x'_1, \dots, x'_{\hat b}$ and $y'_1, \dots, y'_{\hat b}$ playing the roles of $\{x_1, \dots, x_{s}\}$ and $\{y_1, \dots, y_{s}\}$, respectively, we have
\begin{equation*}
    \Prob{\phi(z + w(i) - 1) = f(i) \text{ for all } i\in[s]}\leq (C_{\ref*{lem:graphhamcycles}}/m)^{\hat{b}}.
\end{equation*}
The two inequalities above together imply \ref{claim:boundprob}.

Combining \ref{eqn:cluster-allocation-probability}--\ref{claim:boundprob}, we have
\begin{align*}
    \Prob{\psi(x_i) = y_i~\forall i \in [s]} 
    &\stackrel{\ref{eqn:compatibility}}{=} \sum_{f\in \mathcal{F}}
    \ProbCond{\psi(x_i) = y_i~\forall i \in [s]}{E_f}\Prob{E_f}\\ 
    &\stackrel{\ref{eqn:cluster-allocation-probability}}{\leq}
    \sum_{f\in \mathcal{F}}\ProbCond{\psi(x_i) = y_i~\forall i \in [s]}{E_f} \left(\frac{C'}{n}\right)^s \\
    &\stackrel{\text{\ref{claim:boundprob}}}{\leq} \left(\frac{C'}{n}\right)^s\sum_{f\in \mathcal{F}} \left(\frac{C_{\ref*{lem:graphhamcycles}}}{m}\right)^{\hat{b}} =  \left(\frac{C'}{n}\right)^s \left|\mathcal{F}\right| \left(\frac{C_{\ref*{lem:graphhamcycles}}}{m}\right)^{\hat{b}} \\
    &\stackrel{\ref{eqn:number-compatible}}{\leq} \left(\frac{C'}{n}\right)^s m^{\hat{b}}\left(\frac{C_{\ref*{lem:graphhamcycles}}}{m}\right)^{\hat{b}} 
    = \left(\frac{C'}{n}\right)^s (C_{\ref*{lem:graphhamcycles}})^{\hat{b}}\\
    &\stackrel{\ref{eqn:hat-b-bound}}{\leq} \left(\frac{C'}{n}\right)^s (C_{\ref*{lem:graphhamcycles}})^s= \left(\frac{C'C_{\ref*{lem:graphhamcycles}}}{n}\right)^s,
\end{align*}
as desired. 
\end{proof}

We conclude this section by explaining how to modify the proof of Theorem~\ref{thm:main-vertex-spread} to prove Theorem~\ref{thm:vtx-spread-for-factors}.  First we describe the random embedding algorithm.

\begin{definbox}[Random $F$-factor embedding algorithm]\label{def:factor-distribution}
Let $\alpha>0$, let $k,r \in \mathbb N$, and let $d\in [k - 1]$.
Let $1 / n \ll 1 / C' \ll 1 / C \ll \alpha, 1$ as in Lemma~\ref{lem:randompartitioninglemma}, where $r$ divides $n$ and $C$.
Given a hypergraph $H$ on $n$ vertices with $\delta_d(H)\geq (\delta_{F,d} + \alpha)\binom{n - d}{k-d}$, we define $\psi : G \hookrightarrow H$, a random embedding of an $F$-factor in $H$, as follows.

\textbf{Step 1: Sample random clusters.} 
Let $\mathcal U = \{U_1,\ldots, U_m\}$ be a random partition of $V(H)$ obtained by applying Lemma~\ref{lem:randompartitioninglemma} with $\delta_{F, d}$ playing the role of $\delta$ (the choice of the other variables is as above -- $\eps$ and $t$ are unimportant and can be set to $1$).  

\textbf{Step 2: Find an $F$-factor in each cluster.} 
For each $i\in [m]$, find an $F$-factor in $H[U_i]$.  Label the vertices of $U_i$ as $u_{i,1}, \dots, u_{i, |U_i|}$ such that $H[\{u_{i,(j-1)r+1}, \dots, u_{i, jr}\}] \cong F$ for every $j \in [|U_i| / r]$.

\textbf{Defining $\psi$.} 
Assume without loss of generality that $V(G) = [n]$ and $G[\{(j-1)r+1, \dots, jr\}] \cong F$ for every $j \in [n / r]$.  
For every $i \in [n]$, let $\psi(i) \coloneqq u_{x,y}$, where $x$ is the largest integer such that $i > \sum_{j=1}^{x-1}|U_i|$ and $y = i - \sum_{j=1}^{x-1}|U_i|$.
\end{definbox}
To prove Theorem~\ref{thm:vtx-spread-for-factors}, we need analogues of Lemmas~\ref{lemma:distribution-well-defined} and \ref{lemma:distribution-is-spread}.  The analogue of Lemma~\ref{lemma:distribution-well-defined} is straightforward, so we omit it.  For the analogue of Lemma~\ref{lemma:distribution-is-spread}, we note the following changes: 
\begin{itemize}
    \item since there is no $\phi$ in Definition~\ref{def:factor-distribution}, it can be interpreted as the identity, so $z = 1$;
    \item the family of functions $\cF$ consists of the single function $w$;
    \item we do not need to define $\hat b$ or $b$, and we do not need \ref{eqn:hat-b-bound};
    \item \ref{eqn:cluster-allocation-probability} and \ref{eqn:compatibility} are unchanged, but for \ref{eqn:number-compatible}, we have $|\cF| = 1$, and for \ref{claim:boundprob}, we have simply that this probability is at most $1$.
\end{itemize}
Altogether, we obtain a $(C' / n)$-vertex-spread distribution, as desired.

\section{Spreadness from vertex spreadness}\label{spreadnessfromvertexspreadness}

In this section we prove Proposition~\ref{prop:vtx-spread-implies-spread}, which, combined with Theorem~\ref{thm:main-vertex-spread} and the FKNP theorem, implies Theorem~\ref{thm:loose-cycle}.  We also prove Theorem~\ref{thm:ell>1-cycle} using Theorem~\ref{thm:main-vertex-spread} and Spiro's strengthening of the FKNP theorem, and we prove Theorem~\ref{thm:F-factors-robustness} by combining Theorem~\ref{thm:vtx-spread-for-factors} with the FKNP theorem and a coupling argument of Riordan~\cite{Ri22}.

\subsection{Proof of Proposition~\ref{prop:vtx-spread-implies-spread}}

We begin with the proof of Proposition~\ref{prop:vtx-spread-implies-spread}. 
First we need a lemma, upper bounding the number of ``partial embeddings'' of $G$ into $H$.
\begin{lemma}\label{lemma:partial-injections-bound}
    Let $H$ and $G$ be $n$-vertex $k$-uniform hypergraphs.  If $F \subseteq H$ has $v$ vertices and $c$ components, then there are at most
    \begin{equation*}
        n^c (k\maxdeg(G))^{v - c}
    \end{equation*}
    hypergraph embeddings $\varphi : G[X] \hookrightarrow H[V(F)]$ where $X \in \binom{V(G)}{v}$ and $F \subseteq H[\varphi(X)]$.
\end{lemma}
\begin{proof}
    Let $F_1, \dots, F_c$ be the components of $F$, and for each $i \in [c]$, let $V(F_i) = \{v_{i,1}, \dots, v_{i, |V(F_i)|}\}$, where for every $j \in [|V(F_i)|]\setminus\{1\}$, there exists $j' < j$ such that $v_{i, j}$ and $v_{i, j'}$ are contained in a common edge of $F_i$. 
    Each embedding $\varphi : G[X] \hookrightarrow H[V(F)]$ with $F \subseteq H[\varphi(X)]$ is then determined by 
    \begin{enumerate}
        \item the preimages $\varphi^{-1}(v_{1,1}), \dots, \varphi^{-1}(v_{c, 1})$ of the ``roots'' and
        \item for each $i \in [c]$, a sequence in $[k\maxdeg(G)]^{|V(F_i)| - 1}$, where the $j$th term in the sequence determines $\varphi^{-1}(v_{i,j+1})$ based on $\varphi^{-1}(v_{i,j'})$, where $j' \leq j$ and $v_{i,j'}$ and $v_{i,j+1}$ are contained in a common edge of $F_i$.
    \end{enumerate}
    Note that there are at most $n^c$ choices for the preimages of the roots and at most $(k\maxdeg(G))^{v-c}$ choices for the sequences.  Combining these choices yields the desired bound.
\end{proof}

\begin{proof}[Proof of Proposition~\ref{prop:vtx-spread-implies-spread}]
    Let $1/n\ll 1/C' \ll 1/C, 1/k, 1/\Delta \leq 1$.
    Let $H$ and $G$ be $n$-vertex $k$-uniform hypergraphs, where $\maxdeg(G) \leq \Delta$, and suppose there is a $(C / n)$-vertex-spread distribution $\mu$ on embeddings $G\hookrightarrow H$.  
    For every $F \subseteq H$ isomorphic to $G$, let
    \begin{equation*}
        \mu'(F) \coloneqq \mu\left(\{\varphi :\{\varphi(e) : e \in E(G)\} = E(F) \}\right),
    \end{equation*}
and note that $\mu'$ is a probability distribution on subgraphs of $H$ which are isomorphic to $G$.  

We prove that $\mu'$ is $\left(C' / n^{1/m_1(G)}\right)$-spread.  To that end, let $S \subseteq E(H)$, and let $T \subseteq H$ have edge set $S$ and subject to that, the fewest number of vertices.  
We may assume $T$ is isomorphic to a subgraph of $G$, or else $\mu'(\{F \subseteq H : E(F) \supseteq S\}) = 0$.  
We may also assume $S \neq \emptyset$.
Let $v$ and $c$ be the number of vertices and components of $T$, respectively.  By Lemma~\ref{lemma:partial-injections-bound},
the number of embeddings $\varphi : G[X] \hookrightarrow H[V(T)]$ where $X \in \binom{V(G)}{v}$ and $T \subseteq H[\varphi(X)]$ is at most $n^c(k\Delta^{v - c})$, so since $\mu$ is $(C / n)$-vertex-spread,
\begin{equation*}
    \mu'(\{F \subseteq H : E(F) \supseteq S\}) \leq n^c(k\Delta)^{v-c}\left(\frac{C}{n}\right)^v \leq \left(\frac{(C')^{m_1(G)}}{n}\right)^{v - c} = 
    \left(\frac{(C')^{m_1(G)}}{n}\right)^{|S|(v - c)/|S|}.
\end{equation*}

Let $T_1, \dots, T_c$ be the components of $T$. Since each $T_i$ is isomorphic to a subgraph of $G$, for every $i \in [c]$ we have $|E(T_i)| / (|V(T_i)| - 1) \leq m_1(G)$.  In particular,
\begin{equation*}
    |S| = \sum_{i=1}^c |E(T_i)| \leq m_1(G)\sum_{i=1}^c\left(|V(T_i)| - 1\right) = m_1(G)(v - c),
\end{equation*}
so $(v - c) / |S| \geq 1 / m_1(G)$.  Therefore
\begin{equation*}
    \mu'(\{F \subseteq H : E(F) \supseteq S\}) \leq \left(\frac{(C')^{m_1(G)}}{n}\right)^{|S|(v - c)/|S|} \leq \left(\frac{C'}{n^{1/m_1(G)}}\right)^{|S|},
\end{equation*}
as desired. 
\end{proof}

\subsection{Stronger spreadness: Proof of Theorem~\ref{thm:ell>1-cycle}}
\label{sect:spiro-spread}

Next we prove Theorem~\ref{thm:ell>1-cycle}.  As mentioned, we need a result of Spiro~\cite{Sp21}, which requires the following stronger notion of spreadness.

\begin{definition}\label{def:spiro-spread}
Let $q \in [0, 1]$, and let $r_0, \dots, r_\ell \in \mathbb N$ be a decreasing sequence of positive integers.  Let $(V, \cH)$ be an $r_0$-bounded hypergraph, and let $\mu$ be a probability distribution on $\cH$.  We say $\mu$ is \textit{$(q; r_0, \dots, r_\ell)$-spread} if the following holds for all $i \in [\ell]$:
\begin{equation*}
    \mu\left(\{A \in \cH : |A \cap S| = t\}\right) \leq q^t \text{ for all } t \in [r_i, r_{i-1}] \text{ and } S \in \bigcup_{i=r_i}^{r_{i-1}}\left\{\binom{E}{i} : E \in \cH\right\}.
\end{equation*}
\end{definition}

As noted by Spiro, if $\mu$ is a $q$-spread measure on an $r_0$-bounded hypergraph $(V, \cH)$, then it is also $(4q; r_0, \dots, r_\ell)$ spread where $r_i \coloneqq \lceil r_{i-1} / 2 \rceil$ for $i \in [\ell]$, so the following result with $\ell = \Theta(\log r)$ implies the FKNP theorem for uniform hypergraphs.  (Spiro~\cite[Theorem 3.1]{Sp21} also proved a slightly stronger result in the non-uniform setting which implies the FKNP theorem, but we do not need this result here).

\begin{theorem}[Spiro~\cite{Sp21}]\label{thm:spiro}
    There exists a constant $K_{\ref*{thm:spiro}} > 0$ such that the following holds for all $K > K_{\ref*{thm:spiro}}$.  Let $q \in [0, 1]$, let $r_0, \dots, r_\ell \in \mathbb N$ be a decreasing sequence of positive integers, where $r_\ell = 1$, and let $(V, \cH)$ be an $r_0$-uniform hypergraph.  If there exists a $(q; r_0, \dots, r_{\ell})$-spread probability distribution on $\cH$ and $p \geq K\ell q$, then a $p$-random subset of $V$ contains an edge of $\cH$ with probability at least $1 - K_{\ref*{thm:spiro}} / (K\ell)$.
\end{theorem}

We apply Theorem~\ref{thm:spiro} with $1$ playing the role of $\ell$ in Definition~\ref{def:spiro-spread}.  This special case of Theorem~\ref{thm:spiro} is in some sense the ``base case'' of Spiro's proof~\cite[Lemma 2.5]{Sp21} and can be proved directly with the second moment method.
To apply Theorem~\ref{thm:spiro}, we also need the following analogue of Proposition~\ref{prop:vtx-spread-implies-spread} with this notion of spreadness.

\begin{proposition}\label{prop:ell-cycle-spreadness}
    For every $C > 0$ and $k \in \mathbb N$, there exists $C'=C_{\ref*{prop:ell-cycle-spreadness}}(k, C)>0$ such that the following holds for all $\ell \in \{2, \dots, k-1\}$ and all sufficiently large $n$. 
    Let $H$ be an $n$-vertex $k$-uniform hypergraph. 
    If there is a $(C / n)$-vertex-spread distribution on embeddings $C_{n, k, \ell} \hookrightarrow H$, then there is a $(C'/n^{k-\ell}; n/(k-\ell), 1)$-spread distribution on Hamilton $\ell$-cycles of $H$.
\end{proposition}

Theorem~\ref{thm:spiro} and Proposition~\ref{prop:ell-cycle-spreadness} together imply that $\rdthresh_{\mathcal C_{k,\ell},d} \leq \vsdthresh_{\mathcal C_{k,\ell},d}$ for every $\ell \in \{2, \dots, k - 1\}$ and $d \in [k - 1]$.
Theorem~\ref{thm:ell>1-cycle} follows immediately from Theorems~\ref{thm:spiro} and \ref{thm:main-vertex-spread} and Proposition~\ref{prop:ell-cycle-spreadness}, so the remainder of this subsection is devoted to the proof of Proposition~\ref{prop:ell-cycle-spreadness}.  This proof is inspired by the proof of Kahn, Narayanan, and Park~\cite{KNP21} determining the threshold for $\gnp(n, p)$ to contain the square of a Hamilton cycle.


\begin{lemma}\label{lemma:ell-cycle-subgraph-bound}
   Every connected subgraph of $C_{n,k,\ell}$ with $v$ vertices and $t$ edges satisfies
    \begin{equation*}
        v \geq \min\{(k - \ell)t + \ell, n\}.
    \end{equation*}
    In particular, if $\ell > 0$, then
    \begin{equation*}
        m_1(C_{n,k,\ell}) = \frac{1}{k - \ell}\cdot\frac{n}{n-1}.
    \end{equation*}
\end{lemma}
\begin{proof}
  Let $F \subseteq C_{n,k,\ell}$ be a connected subgraph with $v$ vertices and $t$ edges.  We may assume without loss of generality that  $F$ has vertex set $[v]$ where $v < n$.  For some $I \subseteq [n / (k - \ell)]$ of size $t$, the edge set of $F$ is 
  \begin{equation*}
    \{f_i \coloneqq [k + (k - \ell)(i- 1)]\setminus [(k - \ell)(i - 1)]: i \in I\}.
  \end{equation*}
  Therefore,
  \begin{equation*}
      v \geq \sum_{i=1}^{n / (k - \ell)}\mathds 1_{i\in I}\left|f_i \setminus \bigcup_{j=1}^{i-1}f_j\right| \geq k + (k - \ell)(t - 1) = (k - \ell)t + \ell,
  \end{equation*}
  as desired.  To see that $m_1(C_{n,k,\ell}) = n / ((k - \ell)(n - 1))$ for $\ell > 0$, first note $d_1(C_{n,k,\ell}) = n / ((k - \ell)(n - 1))$.  Moreover, every spanning subgraph $F \subseteq H$ satisfies $d_1(F) \leq d_1(H)$.  If $F \subseteq H$ has fewer than $n$ vertices, then $|V(F)| \geq (k - \ell)|E(F)| + \ell$, so
  \begin{equation*}
      d_1(F) = \frac{|E(F)|}{|V(F)| - 1} \leq \frac{|V(F)| - \ell}{(k - \ell)(|V(F)| - 1)} \leq \frac{1}{k - \ell},
  \end{equation*}
  as desired.
\end{proof}

\begin{lemma}\label{lemma:number-of-subgraphs-bound}
    If $S$ is a subset of edges of $C_{n,k,\ell}$, then the number of subgraphs of $C_{n,k,\ell}$ with $t$ edges, all in $S$, no isolated vertices, and $c$ components, is at most
    \begin{equation*}
        \binom{k|S|}{c}\left(2\cdot 16^k\right)^t.
    \end{equation*}
\end{lemma}
\begin{proof}
    Let $f_i \coloneqq [k + (k - \ell)(i- 1)]\setminus [(k - \ell)(i - 1)]$ for each $i \in [n / (k-\ell)]$ as in Lemma~\ref{lemma:ell-cycle-subgraph-bound}.
    Each subgraph $F\subseteq H$ with $t$ edges, all in $S$, no isolated vertices, and $c$ components $F_1, \dots, F_c$, is determined by 
    \begin{enumerate}
        \item $c$ distinct vertices $\{v_1, \dots, v_c\} \in \binom{\bigcup_{e\in S}e}{c}$ to serve as the ``roots'' of $F_1, \dots, F_c$,
        \item a $c$-composition $t_1, \dots, t_c$ of $t$ (that is, $t_1 + \cdots + t_c = t$), where $|E(F_i)| = t_i$ for all $i \in [c]$, and
        \item for each $i \in [c]$, the edges $E(F_i) \subseteq S \cap \{f_j : |k + (k - \ell)(j - 1) - v_i| \leq 2kt_i\}$.
    \end{enumerate}
    Note that there are at most $\binom{k|S|}{c}$ choices for the roots, at most $\binom{t - 1}{c - 1} \leq 2^t$ choices for the $c$-composition, and at most $\prod_{i=1}^c2^{4kt_i} = \left(16^k\right)^t$ choices for the edges of each $F_i$.  Combining these choices yields the desired bound. 
\end{proof}

\begin{proof}[Proof of Proposition~\ref{prop:ell-cycle-spreadness}]
Let $1/n\ll 1/C' \ll 1/C, 1/k \leq 1$.
Let $H$ be an $n$-vertex $k$-uniform hypergraph, and suppose there is a $(C / n)$-vertex-spread distribution $\mu$ on embeddings $C_{n,k,\ell}\hookrightarrow H$.  Let $\cH$ be the set of edge sets of Hamilton $\ell$-cycles of $H$.
Define a probability distribution $\mu'$ on $\cH$ as in the proof of Proposition~\ref{prop:vtx-spread-implies-spread}, with $C_{n,k,\ell}$ playing the role of $G$.
By Proposition~\ref{prop:vtx-spread-implies-spread} and Lemma~\ref{lemma:ell-cycle-subgraph-bound}, $\mu'$ is $\left(2^{-k}C' / n^{k - \ell}\right)$-spread.

We prove that $\mu'$ is $(C' / n^{k-\ell}; n / (k - \ell), 1)$-spread.  To that end, let $t \in [n/(k-\ell)]$, and let $S \subseteq E(H)$ be a subset of a Hamilton $\ell$-cycle of $H$.  Let $F\subseteq H$ have edge set $S$ and subject to that, the fewest number of vertices.  

If $t \geq n / k$, then since $\mu'$ is $(2^{-k}C' / n^{k-\ell})$-spread, we have
\begin{equation*}
    \mu'\left(\{A \in \cH : |A \cap S| = t\}\right) \leq 2^{|S|}\left(\frac{2^{-k}C'}{n^{k-\ell}}\right)^t 
    \leq \left(\frac{C'}{n^{k-\ell}}\right)^t,
\end{equation*}
as desired.  Therefore, we may assume $t < n / k$, and in particular, every subgraph of $F$ has fewer than $n$ vertices.
By Lemmas~\ref{lemma:partial-injections-bound}, \ref{lemma:ell-cycle-subgraph-bound}, and \ref{lemma:number-of-subgraphs-bound}, since $\mu$ is $(C / n)$-vertex-spread,
\begin{align*}
    \mu'\left(\{A \in \cH : |A \cap S| = t\}\right) &\leq \sum_{c=1}^t \binom{k|S|}{c}(2\cdot 16^k)^t 
    n^c\left(\frac{2k^2C}{n}\right)^{(k - \ell)t + \ell c}\\
    &\leq (2\cdot 16^k)^t\sum_{c=1}^t \left(\frac{ekn\cdot 2k^2C}{c}\right)^c\left(\frac{2k^2C}{n}\right)^{(k - \ell)t + (\ell - 1)c}\\
    &\leq\left(\frac{2\cdot 16^k(2k^2C)^{k-\ell}}{n^{k-\ell}}\right)^t\sum_{c=1}^t\left(\frac{4ek^4C^2}{c}\right)^c\left(\frac{2k^2C}{n}\right)^{(\ell - 2)c}\\
    &\leq \left(\frac{C'/2}{n^{k-\ell}}\right)^t\sum_{c=1}^\infty\left(\frac{1}{c}\right)^{c} \leq \left(\frac{C'}{n^{k-\ell}}\right)^t,
\end{align*}
as desired.
\end{proof}

\subsection{Proof of Theorem~\ref{thm:F-factors-robustness}}
\label{sect:coupling}

We conclude this section with the proof of Theorem~\ref{thm:F-factors-robustness}.
With Theorem~\ref{thm:vtx-spread-for-factors} in hand, the proof of Theorem~\ref{thm:F-factors-robustness} is similar to the PSSS proof of Theorem~\ref{Prop_prev_results}\ref{prev-result-Kr-factor} assuming a similar result for $K_r$-factors \cite[Theorem 1.8]{PSSS22}. 
For hypergraphs $F$ and $H$, the \textit{$F$-complex} of $H$, denoted $H_F$, is the $|V(F)|$-uniform multi-hypergraph with vertex set $V(H)$ in which every copy of $F$ in $H$ corresponds to a distinct hyperedge of $H$ on the same set of vertices.  Note that $H$ has an $F$-factor if and only if $H_F$ contains a perfect matching.  We let $\cG_F(n, p)$ be the binomial random multi-hypergraph on $n$ vertices where every edge of the $F$-complex of the complete hypergraph is included independently with probability $p$.  

To prove Theorem~\ref{thm:F-factors-robustness}, we need following result of Riordan~\cite[Theorem 18]{Ri22}.

\begin{theorem}\label{thm:coupling}
    For every $k, r \in \mathbb N$, there exists $a = a_{\ref*{thm:coupling}}(k, r)\in (0, 1]$ such that the following holds.  
    If $F$ is a strictly $1$-balanced $k$-uniform $r$-vertex hypergraph and $p = p(n) \leq \log^2(n) / n^{1/d_1(F)}$, then for some $\pi = \pi(n) \sim a p^{|E(F)|}$, we may couple $G\sim \cG^{(k)}(n,p)$ with 
    $G_F\sim\cG_F(n,\pi)$ such that, \aas for every $F$-edge present in $H_F$, the corresponding copy of $F$ is present in $G$.    
\end{theorem}

\begin{proof}[Proof of Theorem~\ref{thm:F-factors-robustness}]
    Let $1 / n \ll 1 / C \ll 1 / C' \ll 1 / C' \ll 1/k, 1/r, a, \alpha < 1$, where $a = a_{\ref*{thm:coupling}}(k, r)$.  It suffices to prove the result for $p = Cn^{-1 / d_1(F)}\log^{1 / |E(F)|} n$.
    By Theorem~\ref{thm:coupling}, $H_p$ contains an $F$-factor if $(H_F)_\pi$ contains a perfect matching where $\pi = ap^{|E(F)|} = a C^{|E(F)|} n^{-(|V(F)| - 1)} \log n.$  
    
    By Theorem~\ref{thm:vtx-spread-for-factors}, there exists a $(C'' / n)$-vertex-spread distribution on embeddings of $G \hookrightarrow H$ where $G$ is $n / r$ disjoint copies of $F$.  Since an $F$-factor in $H$ corresponds to a perfect matching in $H_F$, an embedding of $G$ into $H$ is also an embedding of $G'$ into $H'$ where $G'$ and $H'$ are the ``simplifications'' of $G_F$ and $H_F$, respectively ($G'$ is just an $r$-uniform perfect matching because a strictly $1$-balanced hypergraph is connected).
    Hence, Proposition~\ref{prop:vtx-spread-implies-spread} implies that $H'$ supports a $(C' / n^{r - 1})$-spread distribution on its perfect matchings.  Therefore, the Frankston--Kahn--Narayanan--Park theorem implies that $H'_\pi$ contains a perfect matching \aas if $\pi > K C' \log(n / r) / n^{r - 1}$.  Indeed, this inequality holds by the choice of $C$.  Since $H'_\pi$ \aas contains a perfect matching, so does $(H_F)_\pi$.  Thus, as mentioned, by Theorem~\ref{thm:coupling}, $H_p$ \aas contains an $F$-factor, as desired.
\end{proof}

\section{Proof of Proposition~\ref{Prop_hamilton_connectivity}}\label{Sec_hamilton_connectivity}

The following two properties were introduced in \cite{gupta2022general} to codify in abstract what it means for there to be a standard absorption proof for hypergraphs which are obtained from so-called $\mathcal{A}$-chains. In this paper, we are concerned only with hypergraph Hamilton cycles, so the $\mathcal{A}$-chains are obtained by setting $\mathcal{A}$ to be a single edge. Therefore, we state the following properties in the language of hypergraph Hamilton cycles. Note also that in contrast to \cite{gupta2022general}, we use $\binom{n-d}{k-d}$ terms instead of $n^{k-d}$ terms, as this was more convenient in the current paper.

\begin{enumerate}
        \item[\textbf{Ab}]
        For any $\alpha>0$, there exist $0 < \tau, \eta \le \alpha$ and $n_0 \in \mathbb{N}$ so that if $H$ is a $k$-uniform hypergraph on $n\geq n_0$ vertices with $\delta_d(H) \geq (\delta + \alpha)\binom{n-d}{k-d}$, then there exists $A\subseteq V(H)$ of size at most $\tau n$ with the following property. 
        
        For any $L\subseteq V(H)\setminus A$ of size at most $\eta n$ such that $|L|+|A|$ is $(k,\ell)$-path-divisible, there exists an embedding of a $k$-uniform $\ell$-path to $H$ with vertex set $A\cup L$. Furthermore, the embedding of the sets of $\ell$ first and $\ell$ last endpoints of the path does not depend on the subset $L$, meaning that there exist disjoint $\ell$-sets $A_1,A_2\subseteq A$ such that for all $L$, the $\ell$-path covering $A\cup L$ has endpoints $A_1$ and $A_2$.
        
    \item[\textbf{Con}]
        For any $\alpha>0$, there exist a positive integer $c$ and $n_0 \in \mathbb{N}$ so that if $H$ is a $k$-uniform hypergraph $H$ on $n\geq n_0$ vertices with $\delta_d(H) \geq (\delta + \alpha)\binom{n-d}{k-d}$, then the following holds.
            
        For every $S, T \subseteq V(H)$ of vertex-disjoint $\ell$-sets, $H$ contains an embedding of a $k$-uniform $\ell$-path of length at most $c$ with start on $S$ and end on $T$.
\end{enumerate}

As noted in \cite{gupta2022general} (see Section 6, Table 1), both of these properties are known to hold with $\delta_{\mathcal{C}_{\ell,k},d}$ playing the role of $\delta$ for the parameters $\ell,k,d$ as listed in Proposition~\ref{Prop_hamilton_connectivity}(2-3). The following therefore immediately implies Proposition~\ref{Prop_hamilton_connectivity}(2-3).

\begin{proposition}\label{prop:absorption-proof-gives-hamilton-connectivity}
    Let $\delta \in [0, 1]$, let $k \in \mathbb N$, and let $\ell, d \in [k - 1]$.  If $\delta \geq \delta_{\mathcal{C}_{\ell,k},d}$ and the properties \textbf{Ab} and \textbf{Con} hold for $\delta$, $k$, $\ell$, and $d$, then $\hamconnthresh_{k,\ell, d}\leq\delta$.
\end{proposition}
\begin{proof}
    Let $\alpha > 0$, let $n$ be sufficiently large and $(k,\ell)$-path-divisible, and consider a $k$-uniform hypergraph $H$ on $n$ vertices with $\delta_{d}(H)\geq (\delta+\alpha)\binom{n-d}{k-d}$. Let $S$ and $T$ be disjoint vertex subsets of size $\ell$ each. We wish to find a Hamilton $\ell$-path of $H$ with end sets $S$ and $T$.

Let $\eta, \tau, \alpha' \in (0, 1)$ and $c \geq 1$ satisfy $1/n\ll 1/c \ll \eta, \tau \leq \alpha' \ll \alpha$, where $\eta, \tau$ and $c$ satisfy \textbf{Ab} and \textbf{Con}, respectively, with $\alpha'$ playing the role of $\alpha$. Let $R\subseteq V(H)$ be an $(\eta n/2) $-subset so that for any $2\ell$-subset $Q\subseteq V(H)$ we have that $\delta_{d}(H[R\cup Q])\geq (\delta+\alpha/2)\binom{|R\cup Q|-d}{k-d}$ (such a subset exists by a union bound over applications of Lemma~\ref{lem:McDAppl}). 
Since $\alpha' \ll \alpha$, we have $\delta_d(H\setminus R) \geq (\delta + \alpha / 2)\binom{n - d}{k - d}$, so by \textbf{Ab} (with $\alpha'$ playing the role of $\alpha$), there exists an absorbing set $A \subseteq V(H) \setminus R$ of size at most $\tau n$, and $A_1,A_2\subseteq A$ so that for any small $L$, $A\cup L$ has a Hamilton path with $A_1$ and $A_2$ as the first and last $\ell$ vertices of the path. Similarly, $\delta_d(H\setminus (R\cup A)) \geq (\delta + \alpha / 2)\binom{n - d}{k - d}$, so $H\setminus (R\cup A)$ therefore contains two vertex-disjoint $\ell$-paths $P_0$ and $P_1$ together covering all but at most $3k$ vertices outside $R\cup A$.  (To find such paths, first delete a minimal number of vertices from $H\setminus (R\cup A)$ to obtain a hypergraph where the number of vertices is divisible by $k - \ell$. Then, since $\delta \geq \delta_{\mathcal{C}_{\ell,k},d}$, the resulting hypergraph contains a Hamilton $\ell$-cycle, and we again delete a minimum number of vertices from the cycle to obtain the two paths.  Note that in the first step, we delete at most $k$ vertices, and in the second step, we delete at most $2k$ vertices.)  

Now, using \textbf{Con} four times (each time deleting $c$ vertices from the host hypergraph, which does not significantly change the minimum degree condition), we can extend $P_0$ to $P_0'$ so that it has as endpoints $S$ and $A_1$, and extend $P_1$ to $P_1'$ so that it has endpoints $A_2$ and $T$. While doing this, we only use vertices coming from the set $R$ (those which have not been previously used). The remaining vertices in $R$ together with the vertices deleted for divisibility reasons (while finding $P_0, P_1$) can be absorbed into $A$, as the number of remaining vertices is at most $\eta n / 2 + 3k \leq \eta n$ (since $1/n\ll 1/k$) and have the appropriate divisibility property (otherwise $n$ itself would not be $(k,\ell)$-path-divisible). This allows us to find the desired spanning path with endpoints $S$ and $T$.
\end{proof}
\section{Open problems}\label{sec:openproblems}

\subsection{Beyond Hamilton connectivity}

We conjecture that Theorem~\ref{thm:main-vertex-spread} can be strengthened by replacing $\hamconnthresh_{k, \ell, d}$ with $\delta_{\mathcal{C}_{k, \ell}, d}$, as follows.

\begin{conjecture}\label{conj:beyond-ham-conn}
For every $\alpha > 0$ and $k \in \mathbb N$, there exists $C = C_{\ref*{conj:beyond-ham-conn}}(\alpha, k)$ such that the following holds for every $\ell, d \in [k - 1]$ and every sufficiently large $n$ for which $k - \ell$ divides $n$.  If $H$ is an $n$-vertex $k$-uniform hypergraph such that $\delta_d(H) \geq (\delta_{\mathcal{C}_{k, \ell}, d} + \alpha)\binom{n - d}{k - d}$, then there is a $(C / n)$-vertex-spread distribution on embeddings $C_{n,k,\ell} \hookrightarrow H$.
\end{conjecture}
Together with Theorem~\ref{thm:spiro} and Propositions~\ref{prop:vtx-spread-implies-spread} and \ref{prop:ell-cycle-spreadness}, if true Conjecture~\ref{conj:beyond-ham-conn} would imply $\delta_{\mathcal{C}_{k, \ell}, d} = \rdthresh_{\mathcal{C}_{k, \ell}, d} = \vsdthresh_{\mathcal{C}_{k, \ell}, d}$.  
For $\ell = 0$, this indeed holds because $\hamconnthresh_{k, 0, d} = \delta_{\mathcal{C}_{k,0},d}$ for all $k$ and $d$.  However, for $\ell > 0$, this problem seems very difficult.
A natural starting point would be the following case of $k = 3$, $\ell = 2$, and $d = 1$: If $H$ is an $n$-vertex $3$-uniform hypergraph such that $\delta_1(H) \geq (5/9 + \alpha)\binom{n}{2}$, then there is a $(C / n$)-vertex-spread distribution on embeddings $C_{n, 3, 2} \hookrightarrow H$.
Reiher, R\"{o}dl, Ruci\'{n}ski, Schacht, and Szemer\'{e}di~\cite{RRRSS19} proved that $\delta_{\mathcal{C}_{3, 2}, 1} = 5/9$ and noted that $\hamconnthresh_{3,2,1} > 5/9$, so this case does not follow from Theorem~\ref{thm:main-vertex-spread}.

\subsection{Exact minimum-degree thresholds}

It would be interesting to investigate whether Theorem~\ref{thm:main-vertex-spread} holds with the minimum-degree condition replaced by an \textit{exact} one.  However, in general, this problem is challenging because exact minimum-degree conditions for the existence of a single Hamilton $\ell$-cycle are known only in a few special cases.

For example, Katona and Kierstead~\cite{KK99} conjectured that every $n$-vertex $k$-uniform hypergraph $H$ with $\delta_{k-1}(H) \geq \lfloor (n - k + 3) / 2\rfloor$ has a tight Hamilton cycle.  This conjecture was confirmed in the case $k = 3$ for sufficiently large $n$ by R\"{o}dl, Ruci\'{n}ski, and Szemer\'{e}di~\cite{RRZ11} but remains open for $k \geq 4$.  Nevertheless, it seems likely that if the Katona--Kierstead conjecture holds, then \textit{robust} versions of it hold as well.  

More generally, we conjecture that if a $k$-uniform hypergraph has minimum $d$-degree large enough to guarantee a Hamilton $\ell$-cycle, then there is a $(C / n)$-vertex-spread embedding of $C_{n,k,\ell}$ into it, where $C$ only depends on $k$.  
To that end, for every $k\in \mathbb N$, $d \in [k - 1]$, $\ell \in \{0,\dots, k-1\}$, and $n \in \mathbb N$ divisible by $k - \ell$, let $h_{k,\ell,d}(n)$ be the minimum integer $D \geq 0$ such that the following holds: If $H$ is an $n$-vertex $k$-uniform hypergraph satisfying $\delta_d(H) \geq D$, then $H$ has a Hamilton $\ell$-cycle.  
\begin{conjecture}\label{conj:robust-exact-hamilton}
For every $k \in \mathbb N$, there exists $C = C_{\ref*{conj:robust-exact-hamilton}}(k)$ such that the following holds for every $d \in [k - 1]$, $\ell \in \{0,\dots, k-1\}$, and $n \in \mathbb N$ for which $k - \ell$ divides $n$.  If $H$ is an $n$-vertex $k$-uniform hypergraph such that $\delta_d(H) \geq h_{k,\ell,d}(n)$, then there is a $(C / n)$-vertex-spread distribution on embeddings of $C_{n,k,\ell} \hookrightarrow H$.
\end{conjecture}

Note that Conjecture~\ref{conj:robust-exact-hamilton} strengthens Conjecture~\ref{conj:beyond-ham-conn}.
As mentioned, it seems difficult to prove such a result without knowing the value of $h_{k,\ell, d}(n)$.  It would be interesting to confirm Conjecture~\ref{conj:robust-exact-hamilton} for the case $k = 3$ and $d = \ell = 2$, since the result of R\"{o}dl, Ruci\'{n}ski, and Szemer\'{e}di~\cite{RRZ11} implies $h_{3,2,2}(n) = \lfloor n / 2\rfloor$ for all large $n$.  Another interesting case would be $d = k - 1$ and $\ell = 1$, since a result of Han and Zhao~\cite{HZ15} implies $h_{k, 1, k - 1}(n) = \lceil n / (2k - 2)\rceil$ for all large $n$. As mentioned earlier, the robustness result that would be implied by the existence of such a vertex-spread measure has already been confirmed by Anastos, Chakraborti, Kang, Methuku, and Pfenninger \cite{anastos2023}, independently of our work here.

It would also be interesting to investigate whether Theorem~\ref{thm:vtx-spread-for-factors} holds with an exact minimum-degree condition.  Such a result was proved for $d = k - 1$ and $F \cong K^{(k)}_k$ (so an $F$-factor is a perfect matching) by Kang, Kelly, K\"{u}hn, Osthus, and Pfenninger~\cite{KKKOP22} and for $k = 2$ and $F \cong K_r$ by Pham, Sah, Sawhney, and Simkin~\cite{PSSS22}.  For graphs (i.e.\ $k = 2$), minimum-degree conditions for the existence of $F$-factors are comparatively well understood (see e.g.~\cite{KSS01, KO09, KO09survey}), and we suspect many of these results likely admit robust versions.

\subsection{General robustness}

Considering all of the research into robustness to date, it seems that \textit{every} Dirac-type result admits a robust version.  It is plausible that for every family of $k$-uniform hypergraphs $\cF$ and $d \in [k-1]$, we have $\delta_{\cF,d} = \rdthresh_{\cF, d} = \vsdthresh_{\cF, d}$, which would partially explain this phenomenon and significantly generalise Conjecture~\ref{conj:beyond-ham-conn}.  However, proving such a result seems out of reach at present.
\par One consequence of such a result (using the $s=n$ case of the definition of vertex-spread) would be that whenever $H$ has minimum degree large enough to necessarily contain a copy of $F$, $H$ supports at least $(n/C)^n$ many embeddings (as opposed to copies) of $F$, where $n=|V(H)|=|V(F)|$. Although plausible, it seems quite difficult to approach even this special case without any knowledge about what $F$ looks like.

\section*{Acknowledgements}

We thank the referee for their careful reading of this paper and their suggestions.

\bibliographystyle{amsabbrv}
\bibliography{ref}

\end{document}